\documentclass[]{article}
\usepackage{graphicx}
\usepackage{amsmath,amsfonts,amssymb,amsthm,mathtools}
\usepackage{fancyhdr}
\usepackage{titlesec}
\usepackage{indentfirst}
\usepackage{booktabs}
\usepackage{verbatim}
\usepackage{color}

\usepackage{hyperref}
\usepackage{xcolor}

\usepackage{pseudo}

\usepackage{float}
\usepackage{subcaption}
\usepackage[export]{adjustbox}

\usepackage[maxbibnames=99]{biblatex}
\mathtoolsset{showonlyrefs}

\newcommand{\R}{\mathbb{R}}
\newcommand{\g}{\nabla}
\newcommand{\score}{\nabla \log}
\newcommand{\E}{\mathbb{E}}
\newcommand{\F}{\mathcal{F}}
\newcommand{\argmin}{\operatornamewithlimits{argmin}}
\newcommand{\KL}[2]{\text{KL}(#1||#2)}

\newcommand{\cov}{\text{cov}}
\newcommand{\dudt}{\frac{\partial u_t}{\partial t}}

\newcommand{\Au}{A \ast u}
\newcommand{\ustar}{u^*}
\newcommand{\vstar}{v^*}
\newcommand{\sstar}{\score \ustar}

\newcommand{\vbar}{\bar v}

\renewcommand{\d}{\partial}
\newcommand{\spnorm}[1]{\|#1\|} % spectral norm

\usepackage[page,header]{appendix}
\usepackage{titletoc}

\newcommand{\rd}{\,\mathrm{d}}

\numberwithin{equation}{section}
\newtheorem{theorem}{Theorem}[section]
\newtheorem{lemma}[theorem]{Lemma}
\newtheorem{corollary}[theorem]{Corollary}
\newtheorem{proposition}[theorem]{Proposition}
\newtheorem{definition}[theorem]{Definition}
\newtheorem{remark}[theorem]{Remark}
\newtheorem{example}[theorem]{Example}

\begin{document}

\title{Transport based particle methods for \\the Fokker-Planck-Landau equation}

\author{Vasily Ilin\footnote{Department of Mathematics, University of Washington, Seattle, WA, USA. vilin@uw.edu.} 
\quad
Jingwei Hu\footnote{Department of Applied Mathematics, University of Washington, Seattle, WA, USA. hujw@uw.edu.} \quad 
Zhenfu Wang\footnote{Beijing International Center for Mathematical Research, Peking University, Beijing, China. zwang@bicmr.pku.edu.cn.} }
      
\maketitle

\begin{abstract}
    We propose a particle method for numerically solving the Landau equation, inspired by the score-based transport modeling (SBTM) method for the Fokker-Planck equation. This method can preserve some important physical properties of the Landau equation, such as the conservation of mass, momentum, and energy, and decay of estimated entropy. We prove that matching the gradient of the logarithm of the approximate solution is enough to recover the true solution to the Landau equation with Maxwellian molecules. Several numerical experiments in low and moderately high dimensions are performed, with particular emphasis on comparing the proposed method with the traditional particle or blob method.
\end{abstract}

{\small 

}

\section{Introduction}
The Vlasov-Landau equation \cite{Villani02} is a kinetic equation that describes the evolution of the probability density $f(t,x, v)$ of a plasma, with position $x\in \Omega \subset\R^d$ and velocity $v\in \R^d$: 
\begin{equation}
\begin{split}
    &\frac{\d f}{\d t} + v\cdot \g_{x} f + (E + v\times B)\cdot \g_v f 
    = Q(f,f), \label{eqn: Vlasov-Landau}\\
    &Q(f,f)(t,x,v) = \g_v\cdot \int_{\R^d} A(v - w)\left[f(t,x,w)\g_v f(t,x,v) - f(t,x,v)\g_w f(t,x,w)\right]\rd{w}, 
\end{split}
\end{equation}
where $E$ and $B$ are the electric and magnetic fields given externally or determined self-consistently by the Maxwell's equations. The integro-differential nonlinear operator $Q(f,f)$ is the so-called Landau operator which models collisions between charged particles. It can be derived from the Boltzmann collision operator by taking the scattering angle of collisions to be small. The kernel $A$ is a $d\times d$ matrix given by
$A_{i,j}(z) = |z|^\gamma (|z|^2 \delta_{i,j} - z_i z_j)$. The parameter $-d-1\leq \gamma \leq 1$  determines the strength of particle interactions. The most physically relevant case is $\gamma = -d=-3$, which corresponds to Coulomb interactions. The case $\gamma = 0$ corresponds to the so-called Maxwellian molecules, which admits many mathematical simplifications and hence is also widely considered in analysis and computation.

In this work, we will neglect spatial dependence in equation \eqref{eqn: Vlasov-Landau} and focus on the spatially homogeneous Landau equation:
\begin{equation}\label{eqn: Landau equation}
    \frac{\d f}{\d t} 
    = Q(f,f), 
\end{equation}
where $f=f(t,v)$ is a function of $t$ and $v$ only.

A comprehensive study of the Landau equation with Maxwellian molecules was done by Villani in \cite{villani1998spatially, villani2000decrease}. The main results are existence and uniqueness of smooth solutions, conservation of mass, momentum and energy, and decay of entropy and Fisher information. The author also showed how to rewrite the Landau equation in the form of the Fokker-Planck equation with a time-and-space-dependent diffusion matrix. This motivates applying the tools, which are used for the Fokker-Planck equation, to the Landau equation. Also due to this analogy, the Landau equation is often referred to as the Fokker-Planck-Landau equation in the literature. Less is known about the Landau equation with Coulomb interactions, the case when $\gamma = -3$. In fact, its well-posedness remains as an open problem for a long time due to the analogy of the equation to a diffusion-reaction equation. Only recently Guillen and Silvestre proved in \cite{guillen2023landau} that if the initial data is $C^1$, bounded above by a Gaussian function and in particular with finite Fisher information, then there exists a unique strictly positive solution bounded above by a Gaussian. This solution enjoys conservations of mass, momentum and energy, as well as decay of entropy and Fisher information, although the decay rate is not known, unlike the Maxwellian molecules case. 

Numerically approximating the Landau equation also presents a difficult problem due to its high complexity.
It is well-known that grid-based numerical methods such as finite difference/volume methods scale exponentially with dimension, which makes them not efficient for solving the full $6$-dimensional Vlasov-Landau equation. This motivates developing dimension-agnostic numerical methods such as particle methods.

One such method for solving the Fokker-Planck equation was recently proposed by  Shen et al. in \cite{SW22} and  Boffi and Vanden-Eijnden in \cite{boffi2023probability}, termed score-based transport modeling (SBTM) due to its similarity to score-based diffusion modeling. The authors took the gradient flow perspective on the Fokker-Planck equation and proposed a deterministic particle method for solving it using a neural network $s$ to approximate the score $\nabla_v \log f$. In particular, they showed that as the score-matching error 
\begin{equation}
    \int_{\R^d} |s - \nabla_v \log f |^2f\rd{v}
\end{equation}
goes to zero, the KL divergence between the approximate solution and the true solution goes to zero. This gives a theoretical guarantee for using the score-matching method for solving the Fokker-Planck equation. Following this line of thinking, it is clear that the neural network is not necessarily the only choice to approximate the score. In fact, any rich enough family of functions can be used. In this vein,
Maoutsa et al. \cite{maoutsa2020rkhs} used reproducible kernel Hilbert spaces to minimize the score-matching loss. The score matching idea has also been applied to other equations such as the McKean-Vlasov equation \cite{SW23, LWX24}.

Another deterministic particle method that we explore in this work was proposed by Carrillo et al. in \cite{carrillo2019blob, carrillo2020landau}, termed the blob method, first for the Fokker-Planck and porous medium equation and then for the Landau equation. The starting point is to recognize that $\nabla_v \log f$ is the gradient of the functional derivative of the Boltzmann entropy:
\begin{align}
    \nabla_v \log f = \nabla_v \frac{\delta H}{\delta f}, \quad H(f) = \int_{\R^d} f\log f \rd{v}.
\end{align}
If $f$ is approximated by the empirical measure of $N$ particles (Dirac measures), entropy is undefined, so regularized entropy is used instead:
\begin{equation}
    H_\varepsilon(f) = \int_{\R^d} f\log (\phi_\varepsilon \ast f) \rd{v},
\end{equation}
for some symmetric and positive kernel $\phi_\varepsilon$, often taken as the Gaussian. The authors showed that under the proposed numerical method the regularized entropy $H_\varepsilon$ decays over time, mimicking the decay of the entropy $H$ for the true solution.

The main contribution of this work is to show that the SBTM method developed for the Fokker-Planck equation can be used to approximate the Landau equation. Furthermore, the method can preserve some important physical properties of the Landau equation, such as the conservation of mass, momentum, and energy, and the dissipation of estimated entropy. We also obtain a KL/relative entropy bound similar to that for the Fokker-Planck equation in \cite{boffi2023probability,SW23}. We perform comprehensive numerical experiments in different dimensions,  with particular
emphasis on comparing the proposed method with the blob method \cite{carrillo2020landau}.

During the completion process of this manuscript, we became aware of the independent work \cite{HW24}. The method proposed in \cite{HW24} is very similar to ours, with a few key differences which we summarize below:
\begin{itemize}
\item The theoretical result obtained in \cite{HW24} is based on the assumption that the Landau equation is posed on the torus $\mathbb{T}^d$ rather than $\mathbb{R}^d$. This is a restrictive assumption under which many terms become bounded, for example, $|\nabla \log f(t,v)|$ is generally not a bounded function in $\mathbb{R}^d$, even when $f(t, v)$ is a Gaussian distribution. Unlike this work, our proof and method apply to the unbounded domain, which is the realistic setting for the Landau equation. 
\item Our numerical experiments focus on the comparison of the SBTM and the blob method in \cite{carrillo2020landau}. In particular, we demonstrate that for the simple isotropic BKW solution, both methods can perform well with the same number of particles. In fact, it is the anisotropic solution that distinguishes the power of SBTM over the blob method. Moreover, the SBTM method outperforms the blob method in higher dimensions, making it a potentially better choice for solving the full spatially inhomogeneous Vlasov-Landau equation. On the other hand, the work \cite{HW24} mainly considers the test problems from \cite{carrillo2020landau} for which both methods perform well.
\end{itemize}

The outline of this paper is as follows. In section 2 we review some theoretical properties of the Landau equation. These results are useful later to establish the theoretical estimate of the proposed numerical method. The score-based transport modeling (SBTM) method is introduced in section 3. In section 4 we prove a quantitative result justifying the use of the score-matching loss. Section 5 contains the numerical experiments. We conclude in section 6 with an outlook comparing SBTM to the blob method and describing future work.

\section{Properties of the Landau equation}

Our starting point is to write the Landau equation \eqref{eqn: Landau equation} in the form of a continuity equation -- a fundamental tool in the theory of optimal transport \cite{VillaniOT}. If a PDE can be written as a continuity equation, one can obtain its weak solution by solving a system of characteristic ODEs, where each ODE corresponds to a particle. The particle perspective is useful both for intuitive understanding and for numerical methods. See Appendix for a quick review on the continuity equation and its particle solution. 

To be consistent with the optimal transport literature, in the rest of the paper we use the following notation:
\begin{itemize}
\item We use $u_t(x)$ to denote the density function $f(t,v)$, with $x$ instead of $v$ as the independent variable.
\item We normalize $u_t$ to be a probability density, so that $\int_{\R^d} u_t(x) \rd{x} = 1$. Additionally, we do not distinguish between a density and its distribution.
\end{itemize}
In this notation the Landau equation \eqref{eqn: Landau equation} can be written as
\begin{align}
    &\dudt +\g\cdot (v_t u_t) = 0, \label{eqn: Landau equation1}\\
    &v_t(x) = -\int_{\R^d} A(x - y)[\score u_t(x) - \score u_t(y)]u_t(y) \rd y,\\
    &A_{i,j}(z) = |z|^\gamma (|z|^2 \delta_{i,j} - z_i z_j). \label{eqn: Landau collision kernel}
\end{align}

The following proposition combines Proposition 4 in \cite{villani1998spatially}, Theorem 2.3 in \cite{guillen2023landau} and Theorem 1.2 in \cite{guillen2023landau}, where the authors established existence and uniqueness of solutions to the Landau equation, as well as no finite-time blowup.

\begin{theorem}[Existence, uniqueness, and regularity]\label{prop: existence, uniqueness, regularity}
In case of the Maxwellian molecules ($\gamma=0$ in \eqref{eqn: Landau collision kernel}), if the initial data $u_0$ has finite energy
\begin{align}
    \int_{\mathbb{R}^d} |x|^2 \rd u_0(x) < \infty,
\end{align}
the Landau equation \eqref{eqn: Landau equation1} has a unique solution $u_t$ defined for all $t\geq 0$. Moreover, for all $t>0$, $u_t$ is bounded and belongs to $C^\infty(\R^d)$. If $u_0$ is subgaussian, then $u_t$ is subgaussian for all $t$:
\begin{align}
    u_0(x) \leq C_1\exp(-\beta |x|^2) < \infty \implies u_t(x) \leq C_2\exp(-\beta |x|^2) < \infty, \quad C_1,C_2>0.
\end{align}
In case of the Coulomb interactions, $\gamma=-3$ in \eqref{eqn: Landau collision kernel} with $d=3$, if $u_0$ is $C^1$, bounded above by a Gaussian
$$u_0(x) \leq C_0 \exp{(-\beta |x|^2)},$$
and has finite Fisher information
$$\int_{\mathbb{R}^d} |\score u_0|^2 \rd u_0(x) < \infty,$$
then there exists a unique smooth solution $u_t$ that does not blow up in finite time.
\end{theorem}

The next proposition shows that the Landau equation with Maxwellian molecules can be written in the form of the Fokker-Planck equation with a time-and-space-dependent diffusion matrix. Note that this derivation already appeared in \cite{villani1998spatially} but with several rescaling and normalization assumptions. 

First of all, we recall that the solution to the Landau equation \eqref{eqn: Landau equation1} preserves mass, momentum and energy:
\begin{equation}
\int_{\R^d} u_t(x)\rd{x}\equiv 1, \quad V:=\int_{\R^d} x u_t(x)\rd{x}=\int_{\R^d} x u_0(x)\rd{x}, \quad 2E:=\int_{\R^d} |x|^2 u_t(x)\rd{x}=\int_{\R^d} |x|^2 u_0(x)\rd{x}.
\end{equation}
This property can be easily verified using a weak form of the Landau operator. Since our numerical solution also preserves this property, we defer our proof to the later section. We further define the (time-dependent) second order moment by matrix
\begin{equation}
\Sigma(t):=\int_{\mathbb{R}^d} x\otimes x u_t(x) \rd x.
\end{equation}
With the notations $V$, $E$, and $\Sigma$, one can show the following:
\begin{proposition}[Fokker-Planck form]\label{prop: Landau equation in Fokker-Planck form}
    The Landau equation \eqref{eqn: Landau equation1} with Maxwellian molecules $(\gamma=0)$ can be written in the Fokker-Planck form:
    \begin{align}
        &\dudt + \nabla\cdot(v_t u_t ) = 0, \quad v_t(x)=b(x) - D_t(x)\score u_t, \\
        &b(x) = A*\nabla u_t= -(d-1)\left( x - V \right), \quad D_t(x)=A*u_t,
        \end{align}
with the component of the matrix $D_t(x)$ given by
        \begin{align}
        &D_{i,j}(t,x) = \delta_{i,j} \left( |x|^2 + 2E - 2(V\cdot x) \right) + V_i x_j + V_j x_i - x_i x_j - \Sigma_{i,j}(t), \\
        &\Sigma_{i,j}(t) = \Sigma_{i,j}(\infty) - \left( \Sigma_{i,j}(\infty) - \Sigma_{i,j}(0) \right) e^{-4dt}, \\
        &\Sigma_{i,j}(0) = \int_{\mathbb{R}^d} x_i x_j u_0(x) \rd x, \quad \Sigma_{i,j}(\infty) = \frac{1}{d}\left( \delta_{i,j} (2 E - |V|^2) + dV_i V_j \right).
    \end{align}
\end{proposition}
\begin{proof}
    We must show that the velocity field $v_t$ given in \eqref{eqn: Landau equation1} can be written in the form $b - D\score u_t$. First, we split up $v_t$ into two terms:
    \begin{align}
        v_t
        &= -\int_{\R^d} A(x - y)\left[\score u_t(x) - \score u_t(y)\right]u_t(y)\rd y \\
        &= (A\ast(\g u_t)) - (A\ast u_t) \score u_t,
    \end{align}
    where the first term is the drift term and the second term is the diffusion term. The drift term simplifies to
    \begin{equation}
        (A\ast(\g u_t))_i 
        = \sum_j \partial_j A_{i,j} \ast u_t = -(d-1) \left( x_i - V_i \right).
    \end{equation}
    As for the diffusion term, we first compute the convolution $A\ast u_t$:
    \begin{align}
        (A\ast u_t)_{i,j}
        &= \int_{\R^d} \delta_{i,j}|x - y|^2 u_t(y) \rd y - \int_{\R^d} (x_i - y_i)(x_j - y_j) u_t(y) \rd y \\
        &= \delta_{i,j}(|x|^2 + 2E - 2 V \cdot x) + V_i x_j + V_j x_i - x_i x_j - \Sigma_{i,j}. 
    \end{align}
   By a straightforward calculation, one can show that $\Sigma$ satisfies
    \begin{align}
        \Sigma_{i,j}(t) = \Sigma_{i,j}(\infty) - \left( \Sigma_{i,j}(\infty) - \Sigma_{i,j}(0) \right) e^{-4dt}.
    \end{align}
    This completes the proof. 
\end{proof}

\begin{lemma}[Boundedness of $A\ast u$]
\label{lemma: bounds on Maxwell convolution}
    For any distribution $u$ not concentrated on a line, and Maxwellian collision kernel $A(z) = |z|^2 I_d - z\otimes z$, the matrix $A\ast u$ is uniformly bounded from below and bounded from above, i.e., there exists $\varepsilon >0$ such that for all $x\in\R^d$,
    \begin{align}
        \varepsilon \leq \spnorm{A \ast u(x)} \leq 2E + |x-V|^2, \quad 2E = \int_{\R^d} |x|^2 \rd u(x), \quad V = \int_{\R^d} x \rd u(x),
    \end{align}
    where $\spnorm{\cdot}$ denotes the matrix 2-norm, and $\varepsilon$ can be taken to be
    \begin{align}
        \varepsilon = tr(\cov (u)) - \spnorm{\cov (u)}.
    \end{align}
    In particular, the bounds above hold for any continuous density.
\end{lemma}
\begin{proof}
    We denote $x' = x - V$ and compute
    \begin{align}
        A \ast u(x) 
        = I_d (|x'|^2 + 2E - |V|^2) - \left(x'\otimes x' + \Sigma - V\otimes V\right).
    \end{align}
    By triangle inequality, and noting that $\spnorm{x' \otimes x'} = |x'|^2$, we have
    \begin{align}
        \spnorm{A \ast u(x)}
        \geq |x'|^2 + tr(\Sigma - V\otimes V) - |x'|^2 - \spnorm{\Sigma - V\otimes V} 
        = tr(cov(u)) - \spnorm{cov(u)} > 0.
    \end{align}
    The last inequality is strict because $tr(\cov(u))$ is the sum of the eigenvalues of $\cov(u)$, whereas $\spnorm{\cov(u)}$ is the largest eigenvalue. The two quantities are equal if and only if $u$ is concentrated on a line.

    To get the upper bound we recall that $\Sigma$ is positive semi-definite and compute 
    \begin{align}
        \spnorm{A \ast u(x)}
        &\leq |x'|^2 + 2E - |V|^2 - \inf_{|z|=1}\left\{z^T \Sigma z + |x\cdot z|^2\right\} + |V|^2 \leq 2E + |x-V|^2.
    \end{align}
\end{proof}

We will need the following estimate of the score function later on. 
\begin{theorem}\label{thm: score upper bound}
    Let $u_t$ solve the Landau equation with Maxwellian molecules, that is the case $\gamma =0$ in any dimension,  with initial data $u_0 \in W^{2,1}\cap W^{2, \infty}$. Assume further that the initial data $u_0$ satisfies the estimates
    \begin{align}
        u_0(x) \leq C_1 \exp(-|x|^2/C_1), \quad \mbox{and}   \, \quad       |\score u_0(x)|\,  \leq C_2(1 + |x|).
    \end{align}
    Then for any $t\leq T$,  the score function grows at most linearly in $x$, i.e., there exists a universal constant $C>0$ such that
    \begin{align}
        |\score u_t| \leq C(1 + |x| + \sqrt{t}).
    \end{align}
\end{theorem}
\begin{proof}
    This theorem is a major estimate  in the recent paper \cite{FGW}. Combining with the esitmates on $\nabla^2 \log u_t(x)$, the authors can establish entropic propagation of chaos from particle systems towards the Landau equation with Maxwellian molecules. We refer the readers to the proof in \cite{FGW}. 
\end{proof}

\section{Particle methods for the Landau equation}
As outlined in the Appendix, in order to (weakly) solve the continuity equation 
\begin{equation}
    \dudt + \nabla\cdot(v_t u_t) = 0
\end{equation}
 with particles, the standard technique is to assume that the particle solution is given by 
 $$u_t(x) = \frac{1}{n}\sum_{i=1}^n \delta_{X_i(t)}(x),$$
 and integrate the ODEs
\begin{equation}
    \frac{\rd X_i}{dt} = v_t(X_i).
\end{equation}

The challenge to solve the Landau equation \eqref{eqn: Landau equation1} lies in the approximation of the velocity field $v_t$. Since this velocity field mainly depends on the solution $u_t$ through the score $\score u_t$, we aim to find an approximation $s_t(x)$ to the score $\score u_t(x)$. The approximate velocity $v_t$ then becomes
\begin{equation} \label{eqn: velocity field}
  v_t[s_t](x)=  -\int_{\R^d} A(x - y)[s_t(x) -  s_t(y)]u_t(y) \rd y.
\end{equation}
Using particles $\{X_i\}_{i=1}^n$, this is
\begin{equation}
    v_t[s_t](X_i) = -\frac{1}{n}\sum_{j=1}^n A(X_i - X_j)[s_t(X_i) - s_t(X_j)].
\end{equation}

Before addressing the issue of how to estimate the score, we assume at this point that we have at our disposal a sufficiently regular time-dependent vector field $s_t(x)$. We establish the following result.
\begin{proposition}\label{prop: conservations of particle solution}
    Let $s_t(x)$ be any sufficiently regular time-dependent vector field. Consider the continuity equation
    \begin{align}
        &\dudt + \nabla\cdot (v_t u_t) = 0, \\
        &v_t(x) = -\int_{\R^d} A(x - y)[s_t(x) - s_t(y)] \rd u_t(y),
    \end{align}
    with initial density $u_0(x)$.
    \begin{enumerate}
        \item The density solution $u_t(x)$ to the above equation conserves mass, momentum and energy. Furthermore, the estimated entropy decays. That is,
        \begin{align}
             &\int_{\R^d} u_t\rd x = \int_{\R^d} u_0\rd x, \quad \int_{\R^d} xu_t\rd x = \int_{\R^d} xu_0\rd x, \quad \int_{\R^d} |x|^2 u_t\rd x = \int_{\R^d} |x|^2 u_0\rd x,\\
             &\frac{\rd}{\rd{t}} \mathcal{E}_t \leq 0, \quad \mathcal{E}_t := \int_0^t \int_{\R^d} s_\tau(x) \cdot v_\tau(x) u_\tau(x) \rd x \rd \tau.
            \end{align}
        \item The particle solution $\frac{1}{n}\sum_{i=1}^n \delta_{X_i(t)}(x)$ with $\frac{d X_i}{dt} = v_t(X_i)$ to the above equation conserves mass, momentum, and energy. Furthermore, the estimated entropy decays. That is,
        \begin{align}
            &\frac{1}{n} \sum_{i=1}^n X_i(t) = \frac{1}{n} \sum_{i=1}^n X_i(0), \quad \frac{1}{n} \sum_{i=1}^n |X_i(t)|^2 = \frac{1}{n} \sum_{i=1}^n |X_i(0)|^2.\\
            &\frac{\rd}{\rd{t}} \mathcal{E}_t \leq 0, \quad \mathcal{E}_t := \frac{1}{n} \int_0^t \sum_{i=1}^n s_\tau(X_i) \cdot v_\tau(X_i)\rd{\tau}.
        \end{align}
        \item The forward-Euler time-discretized particle solution $X_i^{k+1}  = X_i^k + \Delta t v_t(X_i^k)$ conserves mass and momentum, and decays estimated entropy:
        \begin{align}
            &\frac{1}{n} \sum_{i=1}^n X_i^{k+1} = \frac{1}{n} \sum_{i=1}^n X_i^k,\\
            &\mathcal{E}^{k+1} \leq \mathcal{E}^k,\quad \mathcal{E}^k := \frac{\Delta t}{n} \sum_{j=1}^k \sum_{i=1}^n s(X_i^j) \cdot v(X_i^j).
        \end{align}
        where $\{X_i^k\}_{i=1}^n$ are the particle locations at time $t_k=k\Delta t$.
    \end{enumerate}
\end{proposition}
\begin{proof}
    To prove the first two items it suffices to prove that any weak solution to the continuity equation above conserves mass, momentum and energy, and has a non-positive entropy decay rate. To that end, let $u_t(x)$ be a weak solution and let $\phi: \R^d \to \R$ be a test function. By the definition of a weak solution,
    \begin{align}
        \frac{\rd}{\rd{t}}\int_{\R^d} \phi(x) \rd u_t(x) 
        &= \int_{\R^d} \nabla \phi(x) \cdot v_t(x) \rd u_t(x)\\
        &= -\int_{\R^d} \int_{\R^d} \nabla \phi(x)^T A(x - y) [s(x) - s(y)] \rd u_t(y)\rd u_t(x)\\
        &= \int_{\R^d} \int_{\R^d} \nabla \phi(y)^T A(x - y) [s_t(x) - s_t(y)] \rd u_t(y)\rd u_t(x)\\
        &= -\frac{1}{2}\int_{\R^d} \int_{\R^d} \left[\nabla\phi(x) - \nabla\phi(y)\right]^T A(x - y) [s_t(x) - s_t(y)] \rd u_t(y)\rd u_t(x).
    \end{align}
    Taking $\phi(x) = 1, x, |x|^2$, we get the desired conservations for $u_t$, using the fact that $A(z)z=0$. Taking $\nabla \phi =s_t$  shows the non-positivity of the estimated entropy decay rate
    \begin{align}
        \int_{\R^d} s_t(x) \cdot v_t(x) \rd u_t(x) = 
        -\frac{1}{2}\int_{\R^d} \int_{\R^d} \left[s_t(x) - s_t(y)\right]^T A(x - y) [s_t(x) - s_t(y)] \rd u_t(y)\rd u_t(x) \leq 0,
    \end{align}
    since $A$ is positive semi-definite.

To obtain the conservation of momentum for the time-discretized particle solution, we compute
\begin{align}
    \frac{1}{n}\sum_i X_i^{k+1} &= \frac{1}{n}\sum_i X_i^k - \frac{\Delta t}{n^2} \sum_{i,j} A(X_i^k - X_j^k)[s(X_i^k) - s(X_j^k)] \\
    &= \frac{1}{n}\sum_i X_i^k + \frac{\Delta t}{n^2} \sum_{i,j} A(X_i^k - X_j^k)[s(X_i^k) - s(X_j^k)] \\
    &= \frac{1}{n}\sum_i X_i^k.
\end{align}
The second-to-last line is obtained by interchanging $i$ and $j$ and the symmetry of $A$. This establishes the conservation of momentum. The mass conservation is obvious because each particle carries the same weight $1/n$. To obtain the decay of estimated entropy we compute
\begin{align}
    \mathcal{E}^{k+1} - \mathcal{E}^k 
    &= \frac{\Delta t}{n} \sum_{j=1}^{k+1} \sum_{i=1}^n \left[s(X_i^j) \cdot v(X_i^j)\right] - \frac{\Delta t}{n} \sum_{j=1}^{k} \sum_{i=1}^n \left[ s(X_i^j) \cdot v(X_i^j) \right] \\
    &= \frac{\Delta t}{n} \sum_{i=1}^n \left[s(X_i^{k+1}) \cdot v(X_i^{k+1})\right] \\
    &= -\frac{\Delta t}{n^2}\sum_{i,j=1}^n s(X_i^{k+1})^T A(X_i^{k+1} - X_j^{k+1}) [s(X_i^{k+1}) - s(X_j^{k+1})] \\
    &= -\frac{\Delta t}{2n^2}\sum_{i,j=1}^n \left[s(X_i^{k+1}) - s(X_j^{k+1})\right]^T A(X_i^{k+1} - X_j^{k+1}) [s(X_i^{k+1}) - s(X_j^{k+1})] \leq 0.
\end{align}
\end{proof}

\begin{remark}
Note that above discussion is completely independent of the choice of $s_t(x)$. As long as it is a reasonable approximation to the score $\nabla \log u_t(x)$, one can obtain an approximate solution to the Landau equation, either in the continuous form or in the particle form, such that the conservation property and an estimated entropy decay property would hold. This viewpoint can put a large class of particle methods under the same umbrella, such as the blob method \cite{carrillo2020landau} or the SBTM method to be discussed below.
\end{remark}

\subsection{Score Based Transport Modeling (SBTM)}

We now come to the approximation of the score for the Landau equation. Motivated by the score matching idea due to \cite{hyvarinen2005scorematching, maoutsa2020rkhs} and subsequently used in \cite{boffi2023probability}, we propose to find the approximation $s(x)$ by minimizing the loss
\begin{equation}
\int_{\R^d} (\nabla \log u-s)^T(A*u)(\nabla \log u-s)u\rd{x}:=\int_{\R^d} |\nabla \log u-s|^2_{A*u}u\rd{x},
\end{equation}
where $A$ is the kernel matrix given by \eqref{eqn: Landau collision kernel}.

In fact,
\begin{align}
    \int_{\R^d} |\nabla \log u-s|^2_{A*u}u\rd{x}
    &= \int_{\R^d} \left(|\score u|_{\Au}^2 +|s|_{\Au}^2\right) u\rd{x}  - 2 \int_{\R^d}  s^T (A*u)( \nabla \log u) u\rd{x}\\
    &= \int_{\R^d} \left(|\score u|_{\Au}^2 +|s|_{\Au}^2\right) u\rd{x} - 2\int_{\R^d} s^T (\Au) \g u \rd x\\
    &= \int_{\R^d} \left(|\score u|_{\Au}^2 +|s|_{\Au}^2\right) u\rd{x} + 2\int_{\R^d} u\nabla \cdot ((\Au)s) \rd x,
\end{align}
where we used integration by parts on the second-to-last line. Since the term $\int_{\R^d} |\score u|_{\Au}^2 u\rd{x}$ is a constant in $s$, minimizing the score-matching loss 
\begin{equation}\label{eqn: original score matching loss}
  \int_{\R^d} |\nabla \log u-s|^2_{A*u}u\rd{x}
\end{equation}
is equivalent to minimizing
\begin{equation}\label{eqn: new score matching loss}
  \int_{\R^d} \left(|s|_{\Au}^2  + 2 \nabla \cdot ((\Au)s) \right)u \rd x.
\end{equation}

While $\score u$ and loss \eqref{eqn: original score matching loss} are undefined for a sum of Dirac deltas $\frac{1}{n}\sum_{i=1}^n \delta_{X_i}(x)$, loss \eqref{eqn: new score matching loss} is well-defined and is equal to 
\begin{equation}
    \frac{1}{n}\sum_{i=1}^n \left[s(X_i)^T D(X_i) s(X_i) + 2\nabla \cdot \left(D(X_i) s(X_i)\right)\right], \quad D(X_i) := \frac{1}{n}\sum_{j=1}^n A(X_i-X_j).
\end{equation}
Moreover, if $X_1, \dots, X_n$ are i.i.d. samples from $u$, then by the law of large numbers
\begin{equation}
    \frac{1}{n}\sum_{i=1}^n \left[s(X_i)^T D(X_i) s(X_i) + 2\nabla \cdot \left(D(X_i) s(X_i)\right)\right] \to \E_{u}\left(|s|_{\Au}^2 + 2\nabla \cdot ((\Au) s)\right),
\end{equation}
almost surely as $n\to\infty$. 

Computing $D(X_i)$ for $i=1,\dots,n$ naively results in quadratic complexity $O(n^2)$. One may think that this does not matter since the time integrating step costs $O(n^2)$ operations but in practice time integration is much faster than neural network back-propagation because several gradient descent steps are performed at each time step. To save the computational cost, we use 
the unweighted loss in numerical experiments
\begin{equation}
    \frac{1}{n}\sum_{i=1}^n \left[|s(X_i)|^2 + 2\nabla \cdot s(X_i)\right],\label{eqn: implicit score matching loss}
\end{equation}
and defer the study of the weighted loss to future work. 

Automatic differentiation can be used to compute the divergence of $s$ but for efficiency reasons we use the \textit{denoising trick} following \cite{boffi2023probability}. The denoising trick is a way to approximate the divergence of a vector field, which makes the SBTM method approximately $100$ times faster. As the authors of \cite{boffi2023probability} point out, using the denoising trick prevents overfitting, so even if loss \eqref{eqn: implicit score matching loss} were as cheap to compute as loss \eqref{eqn: loss}, it would still be advisable to use the denoising trick. Lemma B.3 of \cite{boffi2023probability} states that given $Z\sim N(0, I_d)$ and $\alpha>0$,  
\begin{align}
    \lim_{\alpha \to 0} (2\alpha)^{-1}\E\left[ s(x + \alpha Z)\cdot Z - s(x - \alpha Z)\cdot Z \right] = \nabla \cdot s(x). 
\end{align}
Thus, the loss we minimize is
\begin{align} \label{eqn: loss}
    L(s) &= \frac{1}{n}\sum_{i=1}^n \left[|s(X_i)|^2 + \alpha^{-1}(s(X_i + \alpha Z) - s(X_i - \alpha Z))\cdot Z\right].
\end{align}

To approximate $\score u$ well using the minimization problem
\begin{equation}\label{eqn: score matching minimization problem}
    s = \argmin_{s\in\F} L(s),
\end{equation}
the class of functions $\F$ must be rich enough to contain $\score u$. The authors of \cite{boffi2023probability} propose to take $\F$ to be the family of neural networks since they are universal approximators, and to minimize the loss $L(s)$ using gradient descent.

The pseudocode for solving the Landau equation is as follows. 
\begin{pseudo}*
    \hd{SBTM}(u_0, n, t_0, t_{\text{end}}, \Delta t, K) \\
    $t := t_0$ \\
    sample $\{X_i\}_{i=1}^n$ from $u_0$ \\
    initialize NN: $s \approx \arg\min_{s \in \mathcal{F}} \frac{1}{n}\sum_{i=1}^n |s(X_i) - \score u_0(X_i)|^2$ \\
    while $t < t_{\text{end}}$\\+
        $t := t + \Delta t$\\
        Optimize $s$: do $K$ gradient descent steps on $L(s)$ in \eqref{eqn: loss}\\
        for $i = 1, ..., n$\\+
            $X_i := X_i - \frac{\Delta t}{n}\sum_{j=1}^n A(X_i - X_j)[s(X_i) - s(X_j)]$\\--
    output particle locations $X_1,\dots,X_n$
\end{pseudo}
The initialization is done by performing gradient descent steps until the loss is below some pre-specified threshold.

\section{Kullback–Leibler divergence bound}
How good is the score matching approach? In this section, we try to answer this question by quantifying the Kullback–Leibler (KL) divergence between the exact solution of the Landau equation and that of SBTM. The KL divergence or the relative entropy can quantify the distance between two probability densities. It is defined as
\begin{align}
    \KL{u_1}{u_2} = \begin{cases}  \int_{\mathbb{R}^d} \log \frac{u_1}{u_2} \rd u_1(x), \, &\mbox{if\, } \, u_1 \, \mbox{is abosolutely continuous w.r.t. }  u_2, \\
     + \infty,  & \mbox{otherwise}. 
    \end{cases} 
\end{align}
By convexity, one can easily see that $\KL{u_1}{u_2} \geq 0$ and $\KL{u_1}{u_2} = 0$ if and only if $u_1 = u_2$. 

Before analyzing the Landau equation, we review a similar bound from the work of  Boffi and Vanden-Eijnden \cite{boffi2023probability}.

\subsection{KL bound for the Fokker-Planck equation}
\begin{theorem}\label{thm: KL bound in Fokker-Planck}
    Fix $b_t(x), D_t(x)$, such that $D_t$ is a positive semi-definite matrix. Let $s_t(x) \in \R^d$ be sufficiently regular and $u_t$ and $\ustar_t$ satisfy the continuity equations
    \begin{align*}
        &\frac{\partial u^*_t}{\partial t} + \nabla\cdot(v_t^*\ustar_t) = 0, \quad v_t^* = b_t - D_t\score \ustar_t,\\
        &\dudt + \nabla\cdot(v_tu_t) = 0, \quad v_t = b_t - D_ts_t.
    \end{align*}
    Then the KL-divergence is controlled by the score matching loss
    \begin{align*}
      \frac{\rd}{\rd{t}}  \KL{u_t}{u_t^*}  &\leq \frac{1}{2}\int_{\R^d}|s_t - \score u_t|^2_{D_t}u_t \rd x.
    \end{align*}
\end{theorem}
\begin{proof} This follows immediately from the time evolution of the KL/relative entropy between $u_t$ and $\ustar_t$. 
See for instance the proof of Proposition 1 in \cite{boffi2023probability}.
\end{proof}

By combining the above theorem with Proposition \ref{prop: Landau equation in Fokker-Planck form}, we immediately obtain the following corollary.
\begin{corollary}\label{corollary: isotorpic KL bound}
    Let $\ustar_t$ satisfy the Landau equation with Maxwellian molecules
\begin{align}
\frac{\partial \ustar_t}{\d t} + \nabla \cdot (\vstar_t \ustar_t) = 0, \quad \vstar_t(x) = -\int_{\R^d} A(x - y)[\sstar_t(x) - \sstar_t(y)] \ustar_t(y) \rd y,
\end{align}
and let $u_t$ satisfy the continuity equation
\begin{align}
    \dudt + \nabla \cdot (v_t u_t) = 0, \quad v_t(x) = b(x) - D_t(x)s_t(x),
\end{align}
where $b(x)$ and $D_t(x)$ are as in Proposition \ref{prop: Landau equation in Fokker-Planck form}. 
Further, assume that the initial first and second moments match, i.e.,
\begin{align}
    \int_{\R^d} x u_0(x)\rd x = \int_{\R^d} x \ustar_0(x)\rd x, \quad \int_{\R^d} x_i x_j u_0(x) \rd x = \int_{\R^d} x_i x_j\ustar_0(x) \rd x .
\end{align}
Then the time-derivative of the KL-divergence is controlled by the score matching loss
\begin{align*}
  \frac{\rd}{\rd{t}} \KL{u_t}{u_t^*}
   &\leq \frac{1}{2}\int_{\R^d}|s_t - \score u_t|^2_{D_t}u_t \rd x.
\end{align*}
\end{corollary}
\begin{proof}
    Because of conservations, the diffusion matrix $D_t(x)$ and the drift vector $b(x)$ match for all $x,t$. 
    We also need to show that $D_t(x)$ is positive semi-definite. This follows by noting that 
    \begin{align}
        D_t(x) = A\ast \ustar_t(x),
    \end{align}
    and recalling that $A$ is positive semi-definite.
\end{proof}
This can be turned into a numerical method for solving the Landau equation with Maxwellian molecules by using $v_t = b_t - D_ts_t$ as the velocity field in the particle solution. However, this method does not preserve momentum and energy. Furthermore, this numerical scheme heavily relies on rewriting the Landau equation in the Fokker-Planck form, which is not possible for general kernels $A$, and specifically is not possible for the Coulomb kernel. Because of these reasons, we prove a similar inequality to Theorem \ref{thm: KL bound in Fokker-Planck} for the particle solution using the velocity field \eqref{eqn: velocity field}.

\subsection{KL bound for the Landau equation}
Let $\ustar_t$ be a $C^1$ density satisfying the Landau equation \eqref{eqn: Landau equation1} with Maxwellian molecules $(\gamma=0)$
\begin{align} \label{eq: true landau}
    &\frac{\partial \ustar_t}{\d t} + \nabla \cdot (\vstar_t \ustar_t) = 0, \\
    &\vstar_t(x) = -\int_{\R^d} A(x - y) (\sstar_t(x) - \sstar_t(y)) \ustar_t(y) \rd y, 
\end{align}
and let $s_t:\R^d \to \R^d$ be a time-dependent vector field and $u_t$ be a $C^1$ density satisfying the approximate Landau equation
\begin{align} \label{eq: appox landau}
    &\dudt + \nabla \cdot (v_t u_t) = 0, \\
    &v_t(x) = -\int_{\R^d} A(x - y) (s_t(x) - s_t(y)) u_t(y) \rd y.
\end{align}
For both equations, we assume that they subject to the same initial condition $u_0 = \ustar_0$, so $\KL{u_0}{\ustar_0}=0$. 

Since the momentum and energy are conserved in both equations (by properties of the Landau equation and Proposition \ref{prop: conservations of particle solution}), we denote them by $V$ and $E$ respectively:
\begin{align}
    V = \int_{\R^d} x u_t(x) \rd x = \int_{\R^d} x \ustar_t(x) \rd x, \quad 2E = \int_{\R^d} |x|^2 u_t(x) \rd x = \int_{\R^d} |x|^2 \ustar_t(x) \rd x.
\end{align}
Additionally, we denote the (time-dependent) second moment by matrices $\Sigma$ and $\Sigma^*$:
\begin{align}
    \Sigma(t) = \int_{\R^d} x \otimes x u_t(x)\rd x, \quad \Sigma^*(t) = \int_{\R^d} x \otimes x\ustar_t(x)\rd x.
\end{align}

We first establish an intermediate result about the KL divergence.
\begin{proposition}\label{lemma: Landau KL bound}
    The KL divergence between $u_t$ and $\ustar_t$ is controlled by the score-matching loss plus an extra term that depends on the difference between second moments of $u_t$ and $\ustar_t$:
    \begin{align}
        \frac{\rd}{\rd{t}}KL(u_t||\ustar_t) \leq 2\int_{\R^d} | \score u_t - s_t |^2_{A \ast u_t} u_t \rd x + \int_{\R^d} (\score \ustar_t)^T (\Sigma - \Sigma^*) (\score u_t - \score \ustar_t)u_t \rd x.
    \end{align}
\end{proposition}
\begin{proof}
We omit the subscript $t$ throughout this proof to simplify notation. We define an auxillary vector field
\begin{equation}
    \vbar(x) := -\int_{\R^d} A(x - y) (\score u(x) - \score u(y)) u(y) \rd y,
\end{equation}
and compute
\begin{align}
    \frac{\rd}{\rd{t}}KL(u||\ustar) 
    &= -\int_{\R^d}\frac{u}{\ustar} \partial_t \ustar \rd x + \int_{\R^d} \log \frac{u}{\ustar} \partial_t u \rd x\\
    &=\int_{\R^d} \frac{u}{u^*}\nabla \cdot (v^*u^*)\rd{x}-\int_{\R^d} \nabla\cdot (vu)\log \left(\frac{u}{u^*}\right)\rd{x}\\
    &= -\int_{\R^d} \vstar \cdot \nabla \left(\frac{u}{\ustar}\right)\ustar \rd x + \int_{\R^d} v \cdot \score \left(\frac{u}{\ustar}\right) u \rd x\\
    &= \int_{\R^d} ( v-\vstar ) \cdot (\score u - \score \ustar)u \rd x \\
    &= \int_{\R^d} (v-\vbar) \cdot (\score u - \score \ustar)u \rd x+\int_{\R^d} (\vbar-\vstar) \cdot (\score u - \score \ustar)u \rd x . \label{eqn: Landau KL bound decomposition}
\end{align}
We bound the first term in \eqref{eqn: Landau KL bound decomposition} with Young's inequality.
\begin{align}
    &\int_{\R^d} (v-\vbar) \cdot (\score u - \score \ustar)u \rd x\\
    &=\int_{\R^d} (\nabla \log u-s)^T(A*u)\cdot (\nabla \log u-\nabla\log u^*)u\rd{x}-\int_{\R^d} A*(u(\nabla \log u-s))\cdot (\score u - \score \ustar)u\rd{x}\\
    &\leq \int_{\R^d} |\score u - s|^2_{A \ast u} u \rd x + \frac{1}{4} \int_{\R^d} |\score u - \score \ustar|^2_{A \ast u} u \rd x - \int_{\R^d} A\ast(u(\score u - s)) \cdot (\score u - \score \ustar)u \rd x,
\end{align}
where
\begin{align}
    &- \int_{\R^d} A\ast(u(\score u - s)) \cdot (\score u - \score \ustar)u \rd x\\
    &\leq \int_{\R^d}\int_{\R^d} | \score u(y) - s(y) |^2_{A(x-y)}u(x)u(y)\rd x \rd y + \frac{1}{4} \int_{\R^d}\int_{\R^d} | \score u(x) - \score \ustar(x) |^2_{A(x-y)} u(x)u(y)\rd x \rd y \\
    &= \int_{\R^d} | \score u - s |^2_{A \ast u} u \rd x + \frac{1}{4} \int_{\R^d} | \score u - \score \ustar |^2_{A\ast u} u \rd x.
\end{align}
Together we have
\begin{align} \label{eqn: term1}
    \int_{\R^d} (v- \vbar) \cdot \left(\score \frac{u}{\ustar}\right)u \rd x \leq 2\int_{\R^d} | \score u - s |^2_{A \ast u} u \rd x + \frac{1}{2}\int_{\R^d} \left| \score \frac{u}{\ustar} \right|^2_{A\ast u} u \rd x.
\end{align}

Now we analyze the second term in \eqref{eqn: Landau KL bound decomposition}. Using the fact that $A$ is the Maxwellian kernel, we can write
\begin{align}
    \vstar & = -(A\ast \ustar) \score \ustar - (d-1)\left(x - V \right),\\
    \vbar & = -(A\ast u) \score u - (d-1)\left(x - V \right),
\end{align}
so the drift term in $\vbar-\vstar $ cancels out:
\begin{align}
    \vbar-\vstar = (A\ast \ustar) \score \ustar-(A\ast u) \score u.
\end{align}
Thus
\begin{align} \label{eqn: term2}
    & \int_{\R^d} (\vbar-\vstar ) \cdot (\score u - \score \ustar)u \rd x \\
    =& \int_{\R^d} ((A\ast \ustar) \score \ustar-(A\ast u) \score u) \cdot (\score u - \score \ustar)u \rd x \\
    =& -\int_{\R^d} | \score u - \score \ustar |^2_{A \ast u} u \rd x + \int_{\R^d} (\score \ustar)^T (A\ast (\ustar-u)) (\score u - \score \ustar)u \rd x.
\end{align}

Combining \eqref{eqn: term1} and \eqref{eqn: term2}, we have
\begin{align}
    \frac{\rd}{\rd{t}}KL(u||\ustar) 
    \leq 2&\int_{\R^d} | \score u - s |^2_{A \ast u} u \rd x -\frac{1}{2}\int_{\R^d} \left| \score u - \score \ustar\right|^2_{A \ast u} u \rd x \\
    + &\int_{\R^d} (\score \ustar)^T (A\ast (\ustar-u)) \left( \score u - \score \ustar\right)u \rd x,
\end{align}
where the second term is non-positive because $A\ast u$ is positive semi-definite (in fact, strictly positive definite by Lemma \ref{lemma: bounds on Maxwell convolution}). 

Finally, we note that $A\ast(\ustar-u) = \Sigma - \Sigma^*$. This can be seen directly by computing
\begin{align}
    &(A\ast u)_{i,j}(x) = \delta_{i,j}(|x|^2 + 2E - 2V\cdot x) + (x_iV_j + x_jV_i - x_ix_j) - \Sigma_{i,j}(t), \\
    &(A\ast u^*)_{i,j}(x) = \delta_{i,j}(|x|^2 + 2E - 2V\cdot x) + (x_iV_j + x_jV_i - x_ix_j) - \Sigma^*_{i,j}(t),
\end{align}
and subtracting the two equations.
\end{proof}

The second term in Theorem \ref{lemma: Landau KL bound} is hard to control a priori, but if the initial data is isotropic, it vanishes.
\begin{corollary}\label{corollary: isotropic solutions to Landau equation}
    Suppose the initial condition $u_0$ is radially symmetric and $s_t$ is in the form $s_t(x) = x \hat{s}_t(|x|)$ for some $\hat{s}:\R\to\R^d$. Then the KL divergence is controlled only by the score matching loss.
    \begin{align}
        \frac{\rd}{\rd{t}}KL(u_t||\ustar_t) \leq 2\int_{\R^d} | \score u_t - s_t |^2_{A \ast u_t} u_t \rd x.
    \end{align}
\end{corollary}
\begin{proof}[Proof sketch]
    It suffices to prove that $u_t$ and $\ustar_t$ stay isotropic. Indeed, in that case $\Sigma(t)$ and $\Sigma^*(t)$ are scalar matrices and thus are equal because their trace is equal to $2E$, which is conserved. So the difference $\Sigma(t)-\Sigma^*(t)$ in Theorem \ref{lemma: Landau KL bound} is zero.

    So it suffices to show that the velocity fields $v_t$ and $v^*_t$ are radial if $u_t$ and $\ustar_t$ are isotropic. It is a straightforward computation, using the assumed form of $s_t$.
\end{proof}
The assumed form of $s_t$ is not restrictive, because we can simply take $\hat{s}_t:\R \to \R^d$ to be the neural network.

If the initial data is not isotropic, we only get a short time horizon estimate, but $\KL{u_t}{\ustar_t}$ still goes to $0$ as the time-integrated score matching loss goes to zero. 

\begin{theorem}\label{theorem: Landau KL bound}
    Fix $T>0$ and suppose that the following conditions hold:
    \begin{enumerate}
        \item The initial data is in a Sobolev space 
        \begin{align}
            u_0 \in W^{2,1} \cap W^{2,\infty},
        \end{align} 
        and there exists a constant $C_0$ such that
        \begin{align}
            |\score u_0(x)| < C_0(1 + |x|).
        \end{align}
        \item There exist constants $C_0'$ and $\beta > 0$ such that the initial condition is subgaussian
            \begin{align}
                \ustar_0(x) = u_0(x) \leq C_0'e^{-3\beta |x|^2}.
            \end{align}
    \end{enumerate}
    Then for $t\leq T$ small enough such that $\KL{u_t}{\ustar_t} \leq 4$, the following estimate holds
    \begin{align}
        \KL{u_t}{\ustar_t} &\leq 2 e^{Mt} \int_0^t \int_{\R^d} | \score u_\tau - s_\tau |^2_{A \ast u_\tau} u_\tau \rd x \rd\tau,
    \end{align}
    where 
    \begin{align}
M &= d^2 (1 + T) \frac{16C_1}{\beta^2 \varepsilon}(1 + E) \left( 3 + 2\log \left(C_2 \sqrt{\frac{\pi}{\beta}}\right) \right)^2,\\
        \varepsilon &= \inf_{t\leq T}\{tr(\cov(u_t)) - \spnorm{\cov(u_t)}\} > 0,
    \end{align}
    with $C_1, C_2$ being some positive constants.
    In particular, as the time-integrated score-matching loss goes to zero, the KL-divergence goes to 0:
    \begin{align}
        \KL{u_t}{\ustar_t} \to 0, \quad \text{as} \quad \int_0^t \int_{\R^d} | \score u_\tau - s_\tau |^2_{A \ast u_\tau} u_\tau \rd x\rd \tau \to 0.
    \end{align}
\end{theorem}

\begin{proof}
We omit the subscript $t$ for convenience of notation.
Following the proof of Proposition \ref{lemma: Landau KL bound}, there remains to control the term
\begin{align}
 - \frac{1}{2}\int_{\R^d} \left| \score u - \score \ustar\right|^2_{A \ast u} u \rd x + \int_{\R^d} (\score \ustar)^T (\Sigma - \Sigma^*) \left( \score u - \score \ustar\right)u \rd x.
\end{align}
By Lemma \ref{lemma: bounds on Maxwell convolution},
\begin{align}
    A\ast u - \varepsilon I \geq 0.
\end{align}
Using Young's inequality, we get
\begin{align}
    &-\frac{1}{2}\int_{\R^d} \left|\score \frac{u}{\ustar}\right|^2_{A \ast u} u \rd x + \int_{\R^d} (\score \ustar)^T (\Sigma - \Sigma^*) \left(\score \frac{u}{\ustar}\right)u \rd x \\
    \leq& -\frac{1}{2}\int_{\R^d} \left|\score \frac{u}{\ustar}\right|^2_{A \ast u} u \rd x + \frac{\varepsilon}{2}\int_{\R^d} \left|\score \frac{u}{\ustar}\right|^2 u \rd x + \frac{1}{2\varepsilon} \int_{\R^d} \left| (\Sigma - \Sigma^*) \score \ustar \right|^2 u \rd x \\
    \leq& \frac{1}{2\varepsilon} \spnorm{\Sigma - \Sigma^*}^2 \int_{\R^d} \left| \score \ustar \right|^2 u \rd x \\
    \leq& \frac{4C_1}{\varepsilon}(1 + t)(1 + E)\spnorm{\Sigma - \Sigma^*}^2. \label{eqn: term 1 of main proof}
\end{align}
The last inequality (and constant $C_1>0$) follows from Theorem \ref{thm: score upper bound} and conservation of energy.

We rewrite the difference of second moments in the form of a weighted $L^1$ norm.
\begin{align}
    \spnorm{\Sigma - \Sigma^*}^2 
    \leq \sum_{i,j=1}^d \left( \Sigma_{i,j} - \Sigma^*_{i,j} \right)^2 = \sum_{i,j=1}^d \left( \int_{\R^d} x_i x_j ( u -\ustar) \rd x \right)^2 \leq d^2 \left(\int_{\R^d} |x|^2 |u -\ustar| \rd x \right)^2. \label{eqn: term 2 of main proof}
\end{align}
We are now in a position to use the weighted CKP inequality from \cite{bolley2005weighted}
\begin{align}
    \int_{\R^d} \phi(x)|u(x) - \ustar(x)|\rd x \leq \left( 3 + 2\log \int_{\R^d} e^{2\phi(x)}\ustar(x)\rd x \right)\left( \KL{u}{\ustar}^{1/2} + \frac{1}{2}\KL{u}{\ustar} \right).
\end{align}
with $\phi(x) = \beta |x|^2$.
By Theorem \ref{prop: existence, uniqueness, regularity}, there exists a constant $C_2$ such that
$$\ustar(x) \leq C_2 e^{-3\beta |x|^2}.$$
Thus,
\begin{align}
    \int_{\R^d} e^{2\beta|x|^2}\ustar(x)\rd x \leq C_2 \int_{\R^d} e^{-\beta |x|^2}\rd x = C_2 \sqrt{\frac{\pi}{\beta}}. 
\end{align}
We have
\begin{align}
    \beta \int_{\R^d} |x|^2 |u - \ustar| \rd x 
    &\leq \left( 3 + 2\log \int_{\R^d} e^{2\beta |x|^2}\ustar(x)\rd x \right)\left( \KL{u}{\ustar}^{1/2} + \frac{1}{2}\KL{u}{\ustar} \right) \\
    &\leq \beta C_3 \KL{u}{\ustar}^{1/2}, \quad \text{for all $t$ such that $\KL{u}{\ustar}\leq 4$,} \label{eqn: term 3 of main proof}\\
    C_3 &= \frac{2}{\beta}\left( 3 + 2\log \left(C_2 \sqrt{\frac{\pi}{\beta}}\right) \right).
\end{align}

Combining estimates \eqref{eqn: term 1 of main proof}, \eqref{eqn: term 2 of main proof} and \eqref{eqn: term 3 of main proof}, we get the differential inequality that holds for $t \leq T$ small enough that $\KL{u}{\ustar} \leq 4$
\begin{align}
    \frac{\rd}{\rd{t}}\KL{u}{\ustar} 
    &\leq L(t) + \frac{4C_1}{\varepsilon}(1 + E) d^2 C_3^2 (1 + T) \KL{u}{\ustar} \\
    &= L(t) + M \KL{u}{\ustar},
\end{align}
with
\begin{align}
L(t) = 2\int_{\R^d} |s - \score u|^2_{A \ast u} u \rd x,\quad  M = \frac{4C_1}{\varepsilon}(1 + E) d^2 C_3^2 (1 + T).
\end{align}
By generalized Grönwall's inequality we obtain
\begin{align}
    \KL{u}{\ustar} \leq e^{Mt} \int_0^t L(\tau) \rd{\tau}.
\end{align}
\end{proof}

\section{Numerical experiments}
We compare the proposed method with the traditional blob method \cite{carrillo2020landau} on several examples with Maxwellian and Coulomb collision kernels. All simulations and plots were done in Julia programming language with the code available at \href{https://github.com/Vilin97/GradientFlows.jl}{GradientFlows.jl}. The code is well-tested with both unit tests and integration tests running on GitHub Actions. 

All numerical experiments were performed on the Doppio cluster of the Applied Mathematics Department at the University of Washington -- a Silicon Mechanics Rackform R2504.V6 with two 20-core Intel Xeon E5-2698 v4 processors at 2.20GHz, 512 GB of RAM, and two TITAN X (Pascal) GPUs with 12 GB of RAM each.

\subsection{Hyperparameter choice}\label{subsection: hyperparameters}
In the blob method, as the kernel bandwidth $\varepsilon$ goes to either zero or infinity, the velocity of the particles in the blob method goes to zero, which indicates that there is a ``sweet spot" value of $\varepsilon$. The authors of \cite{carrillo2020landau} used the mesh size to set the value of $\varepsilon$. Since we do random initialization instead of grid initialization, we follow the statistical literature \cite{chen2017tutorial} on kernel density estimation and choose $\varepsilon$ to minimize MISE from the gradient of the true solution. We use Silverman's plug-in method:
\begin{align}
    \varepsilon = n^{-1/(d+6)}\left(\sigma_1\cdot...\cdot\sigma_d\right)^{1/d},
\end{align}
where $\sigma_i$ are the eigenvalues of the covariance matrix $\Sigma$ of the sample. We suspect that while this choice is asymptotically optimal for estimating $\nabla u$, it might not be optimal for estimating $\score u = \frac{\nabla u}{u}$.

In the SBTM, there are more hyperparameters, the most important being 
\begin{enumerate}
    \item the NN architecture: number and sizes of layers
    \item learning rate $\eta$
    \item optimizer
    \item activation function
    \item number $K$ of GD steps at each time step
    \item denoising constant $\alpha$
\end{enumerate}
We tuned all of the above and chose $\eta=4\cdot10^{-4}$, the Adam optimizer, the \texttt{softsign} activation function, $K=25$, and $\alpha=0.4$. The choice of $\alpha$ presents a bias-variance trade-off: higher values of $\alpha$ introduce bias but reduce the variance of the approximation of \eqref{eqn: implicit score matching loss}. We use a neural network with two hidden layers, each with 100 neurons, for the BKW example and one hidden layers with 100 neurons for other examples. We use a larger neural network for the BKW example because initial score is nearly singular at 0.

We found that different learning rates work best on different examples, with the most common issue being underfitting of the score in the first few iterations. Choosing the wrong learning rate may lead to underfitting (top right subfigure of Figure \ref{fig: score matching loss underfitting}), and it is hard to select the correct learning rate without having access to the true solution, because the true score matching loss can only be plotted if the true solution is known. Instead, an adaptive strategy may be used: at step 6 of the SBTM algorithm, instead of doing $K$ GD steps, perform GD steps on the denoising loss $L(s)$ in \eqref{eqn: loss} until the implicit loss \eqref{eqn: implicit score matching loss} stops decaying.

The adaptive training strategy sidesteps performing gradient descent on the true implicit loss, which is slow, but uses the true implicit score to detect convergence of the neural network. Figure \ref{fig: score matching loss underfitting} demonstrates how a poor choice of the learning rate $\eta$ and the number of GD steps $K$ can lead to severe underfitting. The black line is the true score matching loss \eqref{eqn: original score matching loss}, the orange line is the true implicit loss \eqref{eqn: implicit score matching loss}, and the blue line is the denoising loss \eqref{eqn: loss}. Each time the particles move, the loss increases, and each step of GD the loss decreases. The adaptive strategy proposed above does well for a wide range of learning rates (compare the two bottom plots), while the non-adaptive strategy of taking $K$ GD steps is very sensitive to the learning rate. When the solution gets closer to the steady state, fewer GD steps are needed to keep a good approximation of the score, which may lead to a significant speed-up if the adaptive strategy is used -- compare the total number of epochs on the horizontal axis of the two subfigures on the left. While the adaptive strategy is promising, we defer its exploration to future work.

For density estimation we use the kernel bandwidth given by the Scott's rule of thumb \cite{scott2015multivariate}:
\begin{align}
    H^* = n^{-1/(d+4)}\Sigma,
\end{align}
where $\Sigma$ is the covariance of the sample. It is asymptotically optimal for minimizing MISE if the underlying distribution is Gaussian. See \cite{chen2017tutorial} for a comprehensive review of multivariate kernel density estimation.
\begin{figure}[H]
    \centering
    \includegraphics[width=0.49\textwidth]{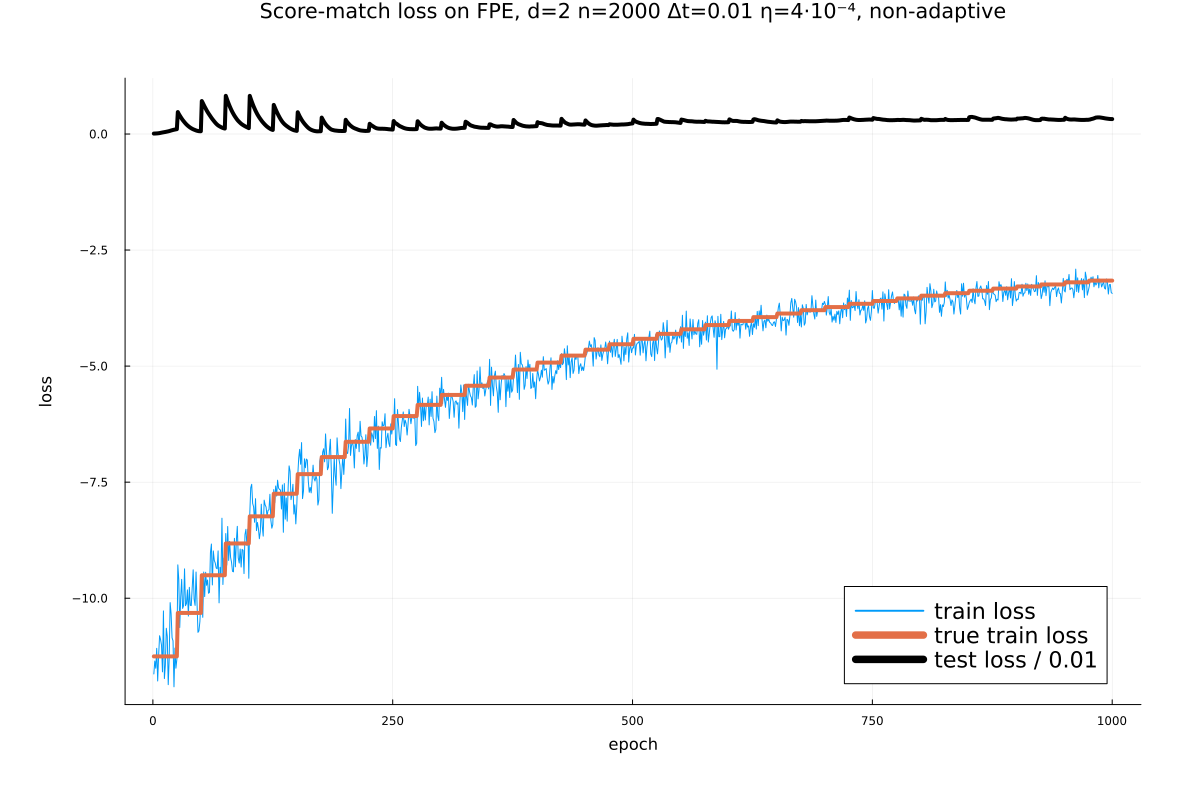}
    \includegraphics[width=0.49\textwidth]{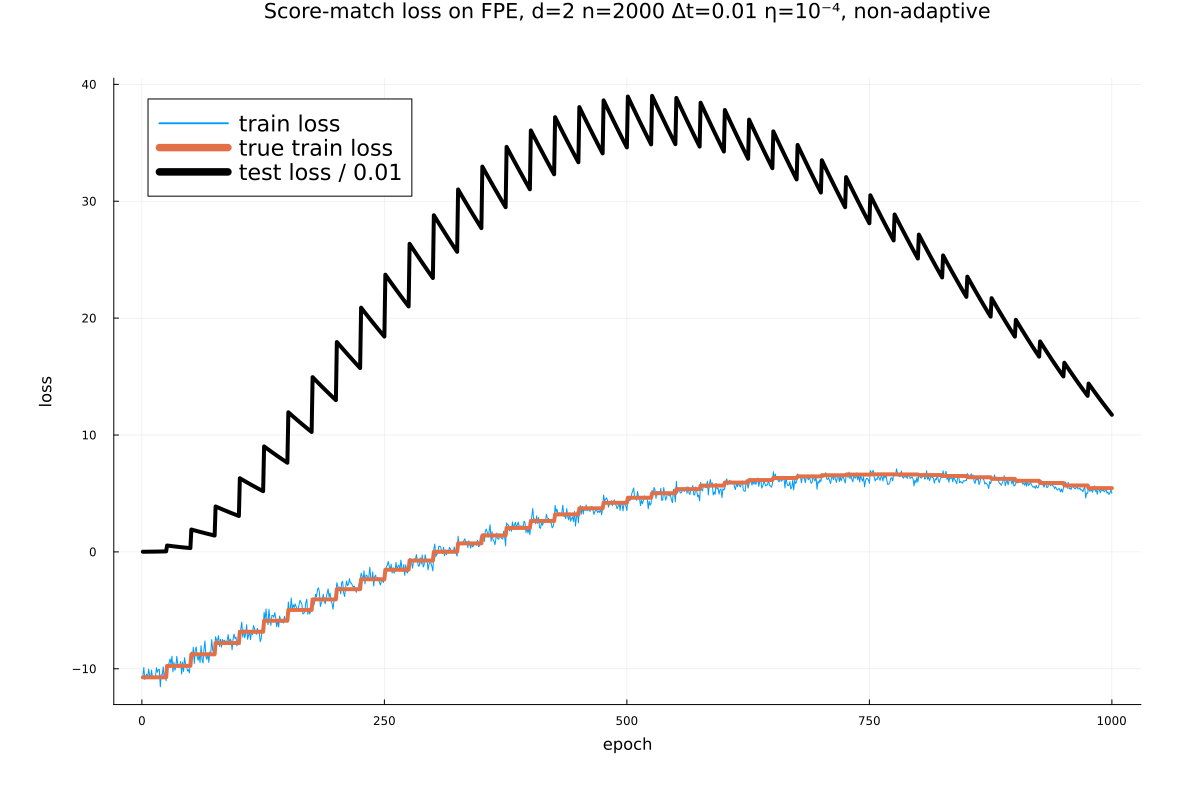}
    \includegraphics[width=0.49\textwidth]{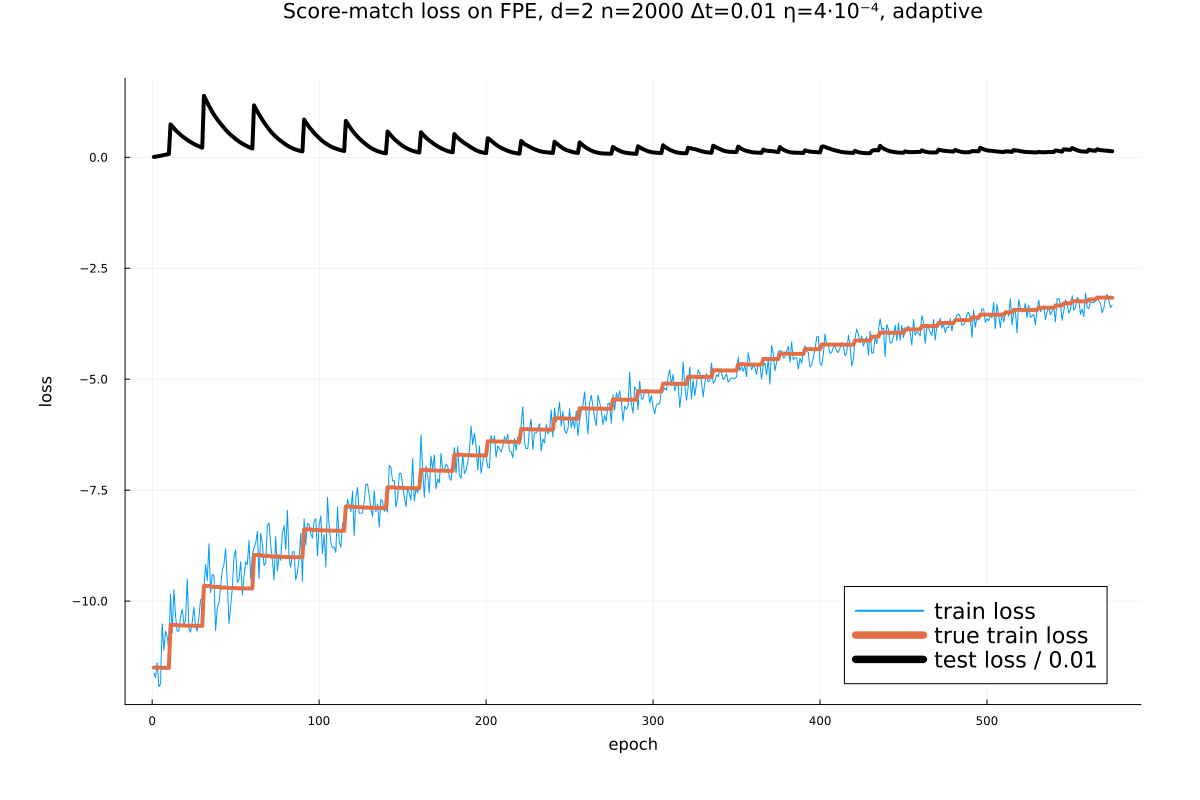}
    \includegraphics[width=0.49\textwidth]{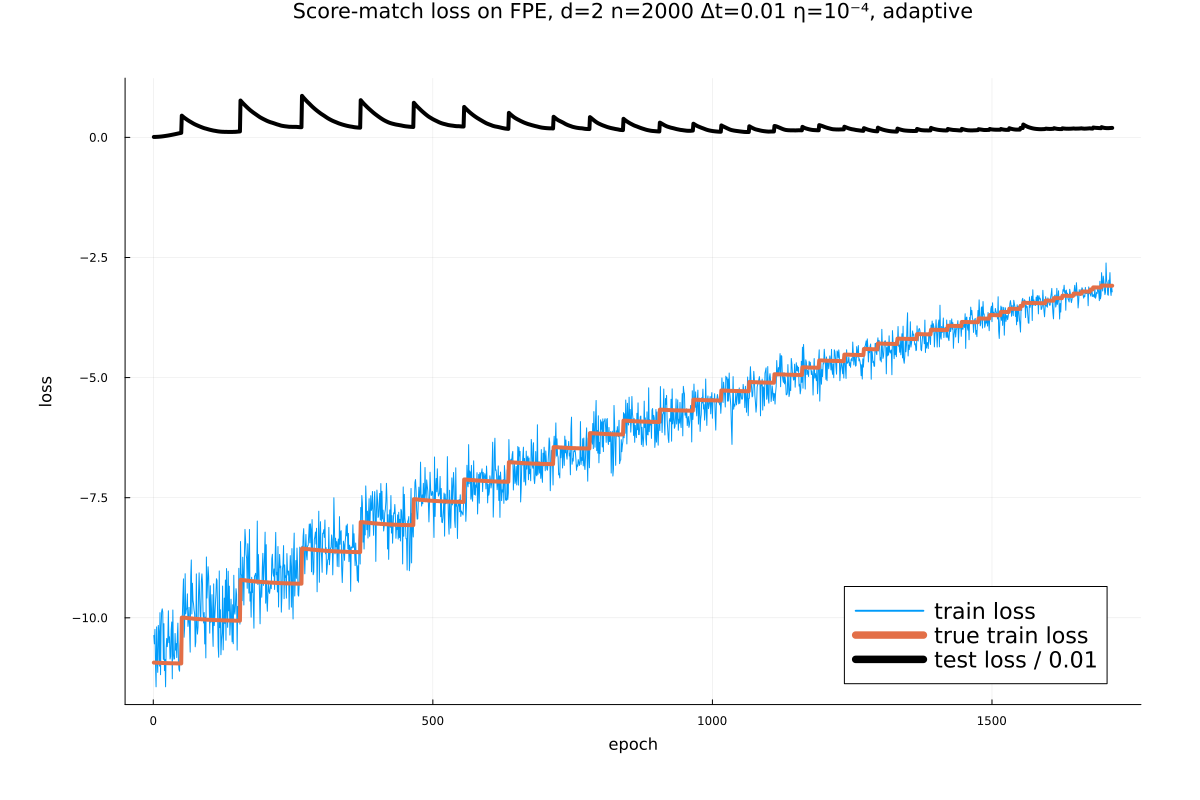}
    \caption{Effect of underfitting, optimal $\eta=4\cdot10^{-4}$ (left) vs low $\eta=10^{-4}$ (right) and non-adaptive training strategy (top) vs adaptive training strategy (bottom).}
    \label{fig: score matching loss underfitting}
\end{figure}

\subsection{Landau equation with Maxwellian kernel}

\begin{example}[BKW solution in dimension $d=3$]\label{example: Landau isotropic BKW solution}
   The BKW solution is an isotropic analytic solution to the Landau equation with Maxwellian kernel. In general dimension $d$, it is given by
        \begin{align}
            A_{i,j}(z) = B(|z|^2\delta_{i,j} - z_i z_j), \quad \ustar_t(x) = \left( 2\pi K \right)^{-d/2}\exp\left( -\frac{|x|^2}{2K} \right)\left(P + Q|x|^2\right),\\
            P = \frac{(d+2)K-d}{2K}, \quad Q = \frac{1-K}{2K^2}, \quad K = 1 - \exp\left( -2B(d-1)t \right), \quad B = 1/24.
        \end{align}
    It is derived in Appendix A of \cite{carrillo2020landau}. We use $t_0=5.5$ to be slightly larger than
    \begin{align}
    t_0^* 
    = \frac{\log (d/2 + 1)}{2B(d-1)}= 5.497\quad \text{for } d=3,\ B=1/24,
    \end{align}
    so that $P(t_0)$ is slightly positive. At $t=t_0^*$ the score is singular. We take $\Delta t = 0.01$ and $t_{\text{end}} = 9.5$.
\end{example}
\begin{remark}\label{remark: isotropic initial data matches second moments}
    Since the BKW solution stays isotropic, second moments are automatically matched, up to the sampling error, on the order of $\frac{1}{\sqrt{n}}$. This is because the trace of the second moment matrix $\Sigma$ is $2E$, and the energy $E$ is conserved. Thus, we do not expect to see a big discrepancy between blob and SBTM in terms of matching the second moments. The simplicity of corollary \ref{corollary: isotorpic KL bound} is another indication of the degeneracy of isotropic initial data for the Landau equation.
\end{remark}
\begin{figure}[htp!]
    \centering
\includegraphics[width=0.49\textwidth]{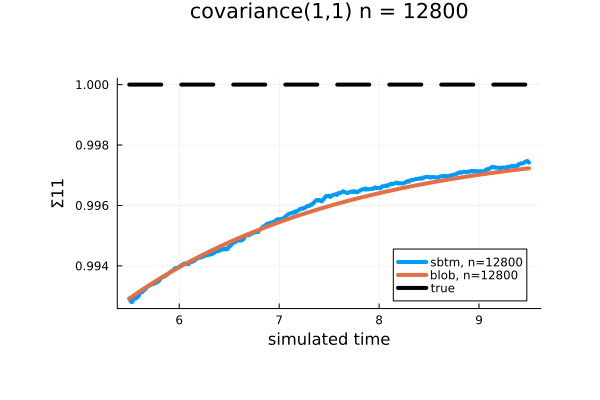}
    \includegraphics[width=0.49\textwidth]{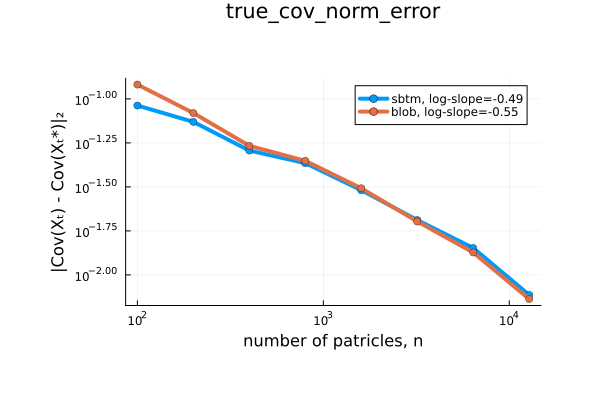}
    \includegraphics[width=0.49\textwidth]{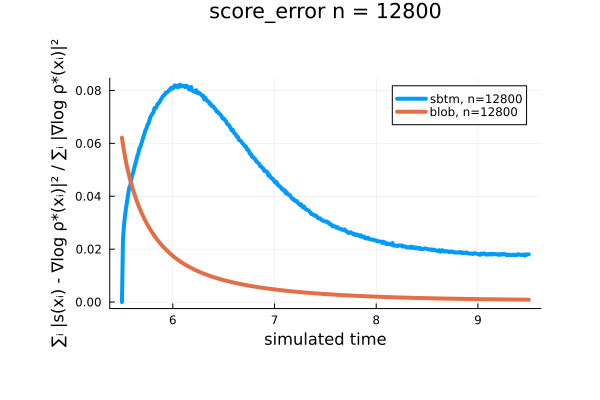}
    \includegraphics[width=0.49\textwidth]{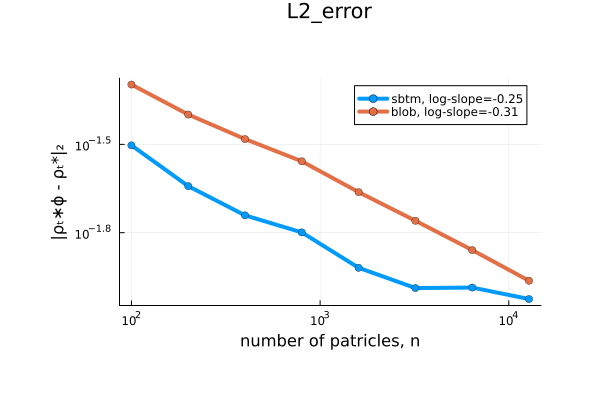}
    \includegraphics[width=0.49\textwidth]{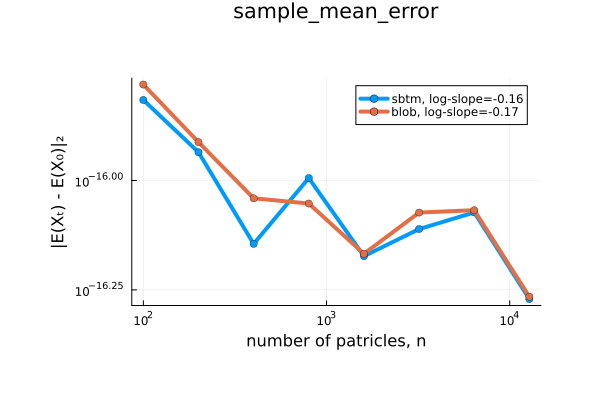}
    \includegraphics[width=0.49\textwidth]{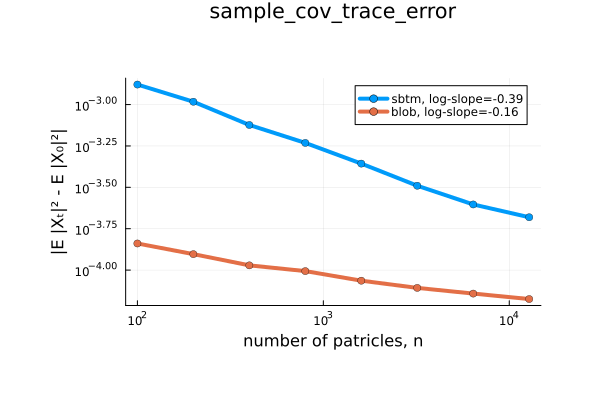}
    \caption{(An isotropic) BKW solution to the Landau equation with Maxwellian kernel in dimension $d=3$. Both the blob method and SBTM do a decent job of approximating the solution.}
    \label{fig: bkw d=3}
\end{figure}

Figure \ref{fig: bkw d=3} shows several metrics that quantify how accurate the numerical solution is. The first plot shows the time evolution of the $(1,1)$ entry of the covariance matrix $\Sigma$ of the numerical solution and of the true solution. The second plot shows the Frobenius norm of the difference between the true second moment matrix $\Sigma^*(t_\text{{end}})$ and the second moment matrix $\Sigma(t_\text{{end}})$ of the numerical solution
\begin{align}
    \sum_{i,j=1}^d |\Sigma_{i,j} - \Sigma^*_{i,j}|^2, \quad \Sigma_{i,j} = \frac{1}{n}\sum_{k} (X_k)_i (X_k)_j,
\end{align}
for particles $\{X_k\}_{k=1}^n$.
The third metric is the normalized true score matching error
\begin{align}
    \frac{\sum_i | \score \ustar_t(X_i) - s(X_i) |^2}{\sum_i | \score \ustar_t(X_i)|^2}.
\end{align}
The last metric is the L2 error of the reconstructed density against the true density
\begin{align}
    \left(\int_{\R^d} |\ustar_t(x) - \frac{1}{n}\sum_j \phi_h(x - X_j)|^2 \rd x\right)^{1/2}.
\end{align}
The $L2$ error relies on the reconstructed density, as opposed to just the particle solution, so this plot may look very different depending on the choice of bandwidth. We also plot the conserved moments. The fifth plot in Figure \ref{fig: bkw d=3} shows the norm of the difference between the sample mean at time $t_0$ and $t_{\text{end}}$
\begin{align}
    \left|\frac{1}{n}\sum_i X_i(t_{\text{end}}) - \frac{1}{n}\sum_i X_i(t_0)\right|.
\end{align}
As predicted by Proposition \ref{prop: conservations of particle solution}, this quantity is identically zero, up to floating point error.
The sixth plot shows the difference between the sample kinetic energy at time $t_0$ and $t_{\text{end}}$
\begin{align}
    \left|\frac{1}{n}\sum_i |X_i(t_{\text{end}})|^2 - \frac{1}{n}\sum_i |X_i(t_0)|^2\right|.
\end{align}
This quantity is not identically zero because of time discretization. But it is on the order of $\Delta t = 0.01$.

As described above, the score matching error is very sensitive to the learning rate. It is likely that the learning rate of $\eta=4\cdot10^{-4}$ is too low for this example but we refrained from optimizing the hyperparameters for each individual example for the sake of a fair comparison.

Per Remark \ref{remark: isotropic initial data matches second moments}, in order to meaningfully assess the quality of the proposed numerical method in terms of matching the second moments, we must use non-isotropic initial data. Without loss of generality, the covariance of the initial data can be taken to be diagonal. Indeed, since the covariance is positive semi-definite, there exists an orthonormal choice of basis in which covariance is diagonal. Furthermore, since the equilibrium distribution of the Landau equation is normal, it is natural to take the initial distribution to be normal with a diagonal covariance. This leads to the following example.
\begin{example}[Anisotropic initial condition in dimension $d=3,10$]\label{example: Landau anisotropic}
    We take the initial distribution to be normal with covariance
    \begin{align}
        \Sigma_{i,j} =\delta_{i,j}\sigma_i,\quad  \sigma_1 = 1.8, \ \sigma_2 = 0.2, \ \sigma_i = 1, \ i=3,\dots,d,
    \end{align}
    so that the trace is equal to $d$, same as in the BKW example. The evolution of the covariance matrix is
    \begin{align}
        &\Sigma^*_{i,j}(t) = \Sigma^*_{i,j}(\infty) - (\Sigma^*_{i,j}(\infty) - \Sigma^*_{i,j}(0))e^{-4dt},\\
       & \Sigma^*_{i,j}(\infty) = \frac{2E}{d}\delta_{i,j}, \quad 2E = tr(\Sigma(0)).
    \end{align} 
    We use $\Delta t = 0.01$, and $t_{\text{end}} = 4$.
\end{example}
Since the true solution is not known, we cannot plot the L2 error or the score error. However, there exists a closed form solution for the second moments of the true solution, which we can compare to the second moments of the numerical solution. Figure \ref{fig: anisotropic landau d=3} shows the covariance trajectories, with the dashed line denoting the analytic solution, and the Frobenius norm of the difference of second moments. The fourth subfigure is the estimated entropy production/decay rate
\begin{align}
    \frac{1}{n}\sum_{i=1}^n s_t(X_i)\cdot v_t(X_i).
\end{align} 
The covariance and entropy rate plots use $n=100$ and $n=12800$ particles. SBTM matches covariance well even with $n=100$ particles, while the blob method needs a lot more particles to achieve the same accuracy. Using the entropy decay trajectories with $n=12800$ as a reference, SBTM with $n=100$ matches it well, unlike the blob method.
\begin{figure}[htp!]
    \centering
  \includegraphics[width=0.49\textwidth]{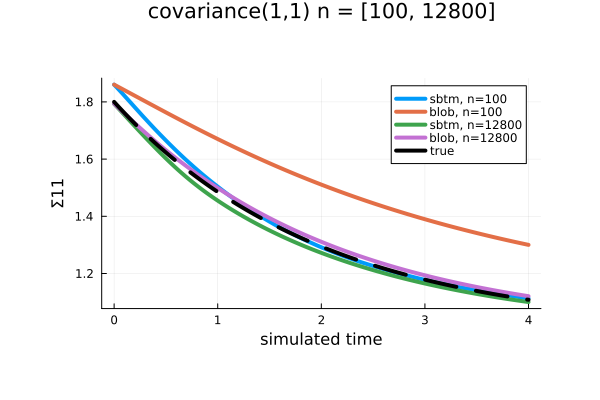}
    \includegraphics[width=0.49\textwidth]{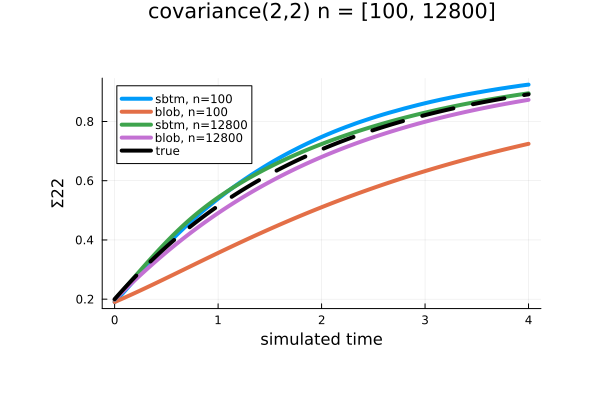}
    \includegraphics[width=0.49\textwidth]{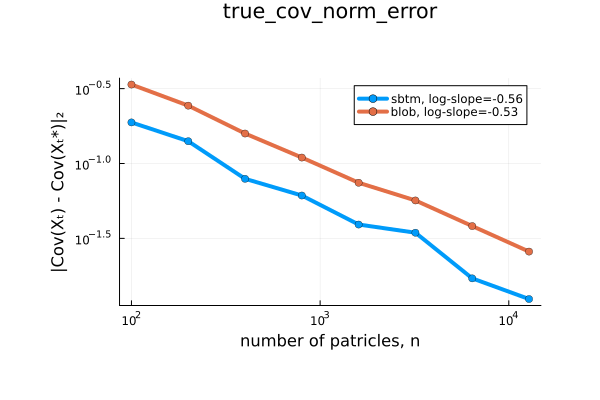}
    \includegraphics[width=0.49\textwidth]{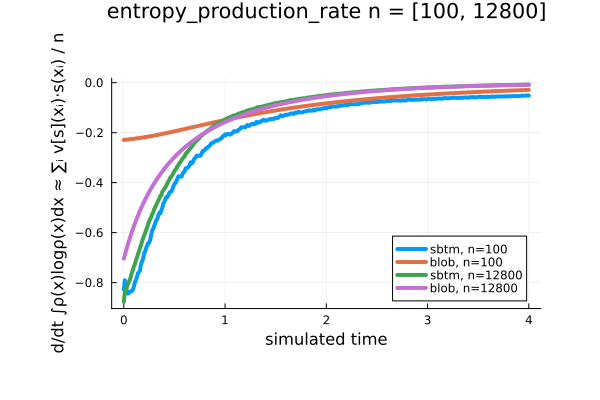}
      \caption{(An anisotropic) solution to the Landau equation with Maxwellian kernel in dimension $d=3$. SBTM can match the covariance well even with $n=100$ particles, while the blob methods needs a lot more particles to achieve the same accuracy.}
    \label{fig: anisotropic landau d=3}
\end{figure}

Although the (physically relevant) Landau operator resides in three-dimensional velocity space, when it is coupled with the full Vlasov operator the ambient space is six dimension. This high dimensional problem causes a lot of numerical difficulties for the traditional particle method. In the recent work \cite{BCH24}, the blob method is generalized to the Vlasov-Landau equation. Due to the complexity of the method, the highest dimension treated there is 1D2V (one dimension in the physical space and two dimension in the velocity space). For this reason, we also compare the blob method and SBTM in dimension $d=10$. The purpose is to examine their performance for higher dimensional problem. It is evident from Figure \ref{fig: anisotropic landau d=10} that even $n=25600$ particles are not enough for the blob method to match the true covariance trajectory, whereas SBTM gives a very good prediction.
\begin{figure}[htp!]
    \centering
    \includegraphics[width=0.49\textwidth]{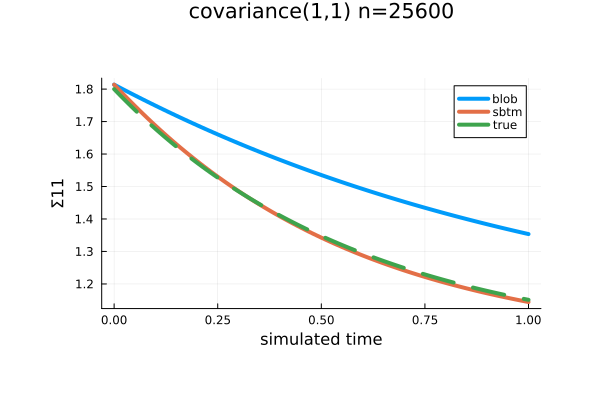}
    \includegraphics[width=0.49\textwidth]{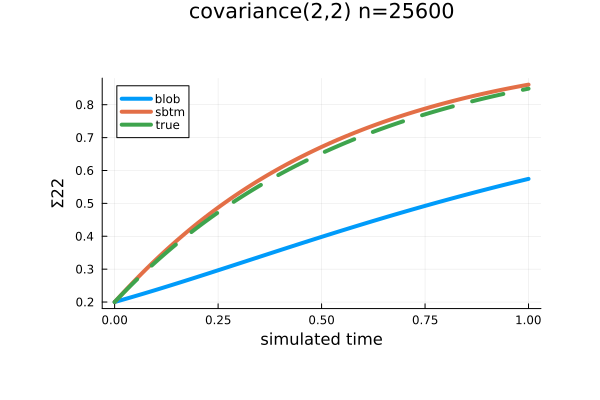}
    \includegraphics[width=0.49\textwidth]{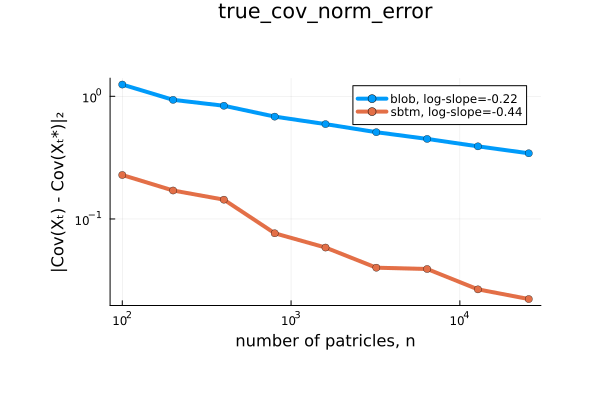}
    \caption{(An anisotropic) solution to the Landau equation with Maxwellian kernel in dimension $d=10$. Even $n=25600$ particles are not enough for the blob method to match the true covariance trajectory, whereas SBTM gives a very good prediction.}
    \label{fig: anisotropic landau d=10}
\end{figure}

\subsection{Landau equation with Coulomb kernel}
\begin{example}[Anisotropic initial condition in dimension $d=3$]
    We take the initial distribution to be normal with covariance
    \begin{equation}
        \Sigma_{i,j} = \delta_{i,j}\sigma_i, \quad \sigma_1 = 1.8, \ \sigma_2 = 0.2, \ \sigma_3 = 1,
    \end{equation}
    so that the trace is 3, and use the Coulomb collision kernel
    $A(z) = \frac{1}{|z|^3}(|z|^2 I_d - z\otimes z$). We choose $\Delta t = 1$, and $t_{\text{end}} = 300$. We do not observe any instabilities with such a large time step.
\end{example}
\begin{figure}[htp!]
    \centering
 \includegraphics[width=0.49\textwidth]{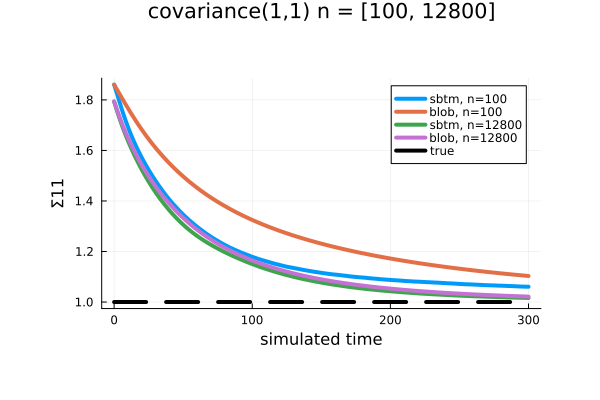}
    \includegraphics[width=0.49\textwidth]{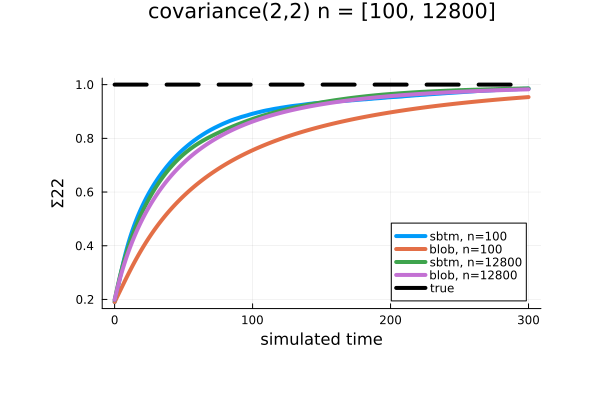}
    \includegraphics[width=0.49\textwidth]{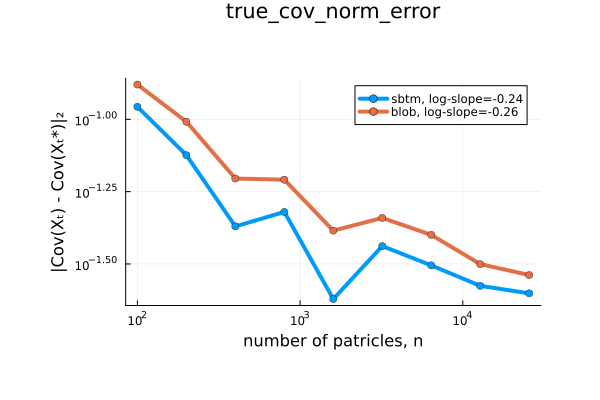}
    \includegraphics[width=0.49\textwidth]{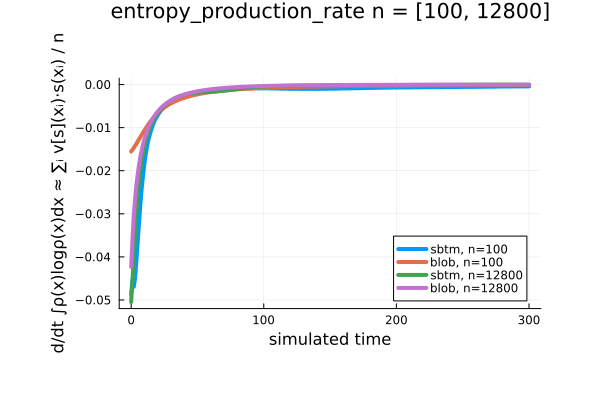}
       \caption{(An anisotropic) solution to the Landau equation with Coulomb kernel in dimension $d=3$. SBTM achieves better accuracy with $n=100$ particles in terms of matching covariance and entropy decay rate of the simulations with $n=12800$ particles.}
    \label{fig: coulomb landau d=3}
\end{figure}
While the true covariance and entropy decay rate trajectories are unknown as well as all other statistics of the true solution, if the trajectories of the SBTM and the blob method with large $n$ coincide, that trajectory is likely the correct one and serves as the reference for the trajectories with low $n$. In Figure \ref{fig: coulomb landau d=3} the first two plots are the diagonal covariance entries $\Sigma_{1,1}$ and $\Sigma_{2,2}$, with the black dashed line being the covariance of the steady state. The third plot is the Frobenius norm of the difference between the steady-state second moment matrix $I_3$ and the second moment matrix of the particle solution at $t=t_{\text{end}}$. The fourth plot is the estimated entropy decay rate
\begin{align}
    \frac{1}{n}\sum_{i=1}^n s_t(X_i)\cdot v_t(X_i).
\end{align}
The horizontal axis of first covariance and entropy rate plots is time. SBTM achieves better accuracy with $n=100$ particles in terms of matching covariance and entropy decay rate of the simulations with $n=12800$ particles.

\subsection{Runtime}
SBTM has linear in the number of particles computational complexity of the score estimation routine, which is the computational bottleneck. Additionally, it is easy to use the GPU to speed up the neural network training using of-the-shelf libraries, such as \texttt{Flux.jl} in \texttt{Julia}. We did this for the heat equation, where we observed a roughly $O(n^{0.6})$ scaling of the computational time on the GPU for the SBTM method.

The blob method has quadratic asymptotic complexity. It does not benefit as much from parallelization as SBTM, and has a quadratic computational time scaling.

Figure \ref{fig: timings} shows the timings of the two methods as a function of the number of particles $n$, on the CPU and GPU. The GPU timing is for the heat equation but the score estimation routine is the bottleneck for all the equations we have tested, so most of the computational time is score estimation.
\begin{figure}[htp!]
    \centering
    \includegraphics[trim={5mm 5mm 0 10mm}, clip, width=0.49\textwidth]{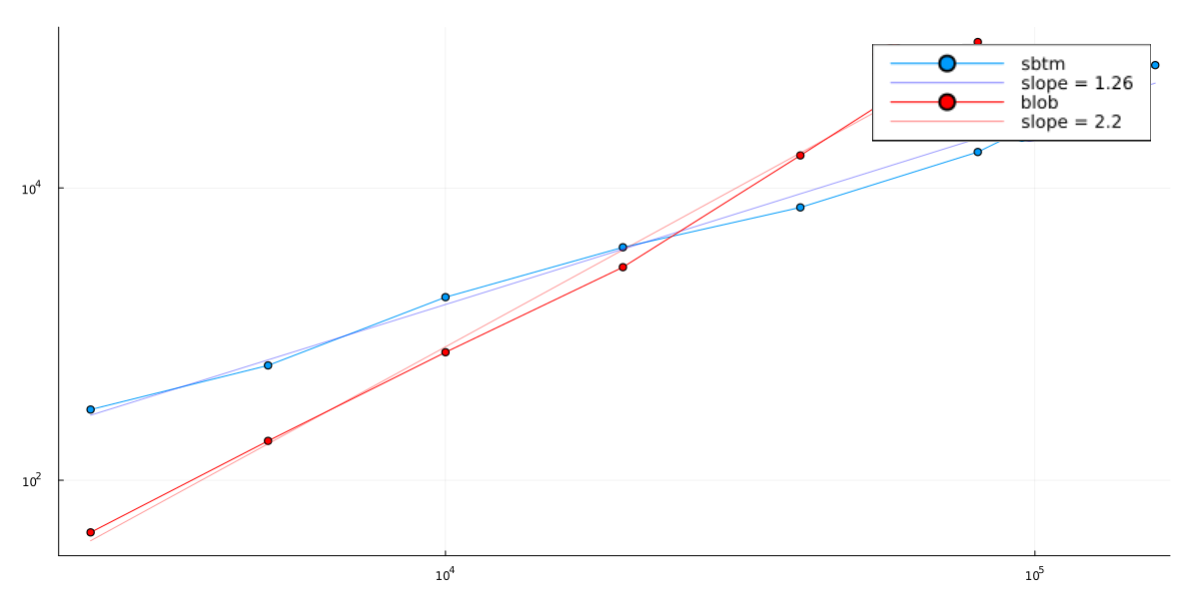}
    \includegraphics[trim={5mm 5mm 0 10mm}, clip, width=0.49\textwidth]{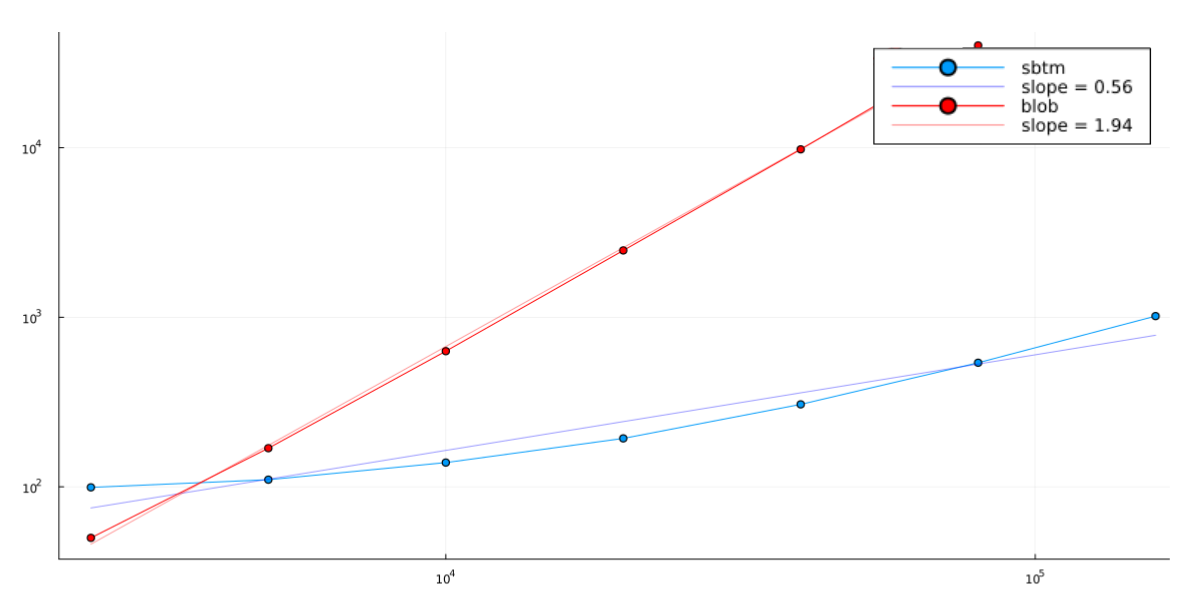}
     \caption{Computational time in seconds (vertical) vs number of particles $n$ (horizontal), on the CPU with 10 threads (left) and GPU (right) log-log scale.}
    \label{fig: timings}
\end{figure}

\section{Conclusion}
In this paper, we proposed a particle method for the Fokker-Planck-Landau equation using the score-based-transport-modeling (SBTM). Compared with the traditional blob method \cite{carrillo2020landau}, where only one parameter $\varepsilon$ (regularization parameter for the entropy) is used to approximate the score, SBTM uses the neural network to approximate the score. We also obtained a theoretical guarantee of the method by showing that matching the score of the approximate solution is enough to recover the true solution to the Landau equation with Maxwellian molecules.
Through a series of numerical examples, we demonstrated that SBTM has the following advantages over the blob method:
\begin{enumerate}
\item offers better accuracy especially for anisotropic solutions;
    \item offers better accuracy in the sparse particles regime, e.g., high dimension or small number of particles in moderate dimension;
    \item linear in the number of particles runtime scaling for score approximation;
    \item easy parallelization on the GPU using off-the-shelf ML libraries, further reducing the runtime.
\end{enumerate}
The main disadvantage of the SBTM method is the large number of hyperparameters, including neural network architecture, learning rate, and the number of gradient descent steps. The adaptive training strategy outlined in section \ref{subsection: hyperparameters} can be more forgiving to a bad choice of hyperparameters.

Regarding the ongoing and future work: on the theoretical side, we are considering the extension of Theorem \ref{theorem: Landau KL bound} to other values of $\gamma$, with the Coulomb kernel being of particular interest. On the numerical side, based on our preliminary analysis, making the number of gradient descent steps adaptive can greatly improve accuracy and reduce runtime. Lastly, SBTM can be used as a kernel to solve the full Vlasov-Landau equation \eqref{eqn: Vlasov-Landau} in a similar vein as in \cite{BCH24}. Deriving any theoretical guarantees for this method and performing numerical experiments is an important direction for future work.

\section{Appendix: The continuity equation and its particle solution}
    In this review technical assumptions are omitted but can be found in \cite{santambrogio2015optimal}. Let $v_t(x) \in \R^d$ be a time-dependent vector field. Equation 
\begin{equation}\label{eqn: continuity equation}
    \dudt + \g \cdot (v_t u_t) = 0
\end{equation}
is called a \textit{continuity equation}. A weak solution of the continuity equation can be defined the same way as for any PDE but there is a simpler alternative definition. See Proposition 4.2 in \cite{santambrogio2015optimal} for a proof that the definition below is equivalent to the usual definition.

\begin{definition}[Weak solution of continuity equation]
    A weak solution of the continuity equation \eqref{eqn: continuity equation} is a curve of probability distributions $u_t$ indexed by $t\geq 0$ such that for all test functions $\phi \in C_c^\infty(\R^d)$ (compactly supported smooth functions), the following holds:
    \begin{align}\label{eqn: weak solution}
        \frac{\rd}{\rd{t}}\int_{\R^d} \phi(x) \rd u_t(x) = \int_{\R^d} \g \phi(x) \cdot v_t(x) \rd u_t(x).
    \end{align}
\end{definition}

The next proposition shows how to obtain a weak solution by solving a system of ODEs.
\begin{proposition}
    The particle solution given by the empirical measure as below
    \begin{align}
        u_t(x) = \frac{1}{n}\sum_{i=1}^n \delta_{X_i(t)}(x), \quad \frac{\rd X_i}{\rd{t}} = v_t(X_i)
    \end{align}
    is a weak solution of the continuity equation \eqref{eqn: continuity equation}.
\end{proposition}
\begin{proof}
    Let $\phi(x)$ be a test function. Then
    \begin{align}
        \frac{\rd}{\rd{t}}\int_{\R^d} \phi(x) \rd{u_t(x) }
        = \frac{\rd}{\rd{t}} \frac{1}{n}\sum_i \phi(X_i(t)) = \frac{1}{n}\sum_i \g \phi(X_i(t)) \cdot \frac{\rd{X_i}}{\rd{t}} = \int_{\R^d} \g \phi(x) \cdot v_t(x) \rd{u_t(x)}.
    \end{align}
\end{proof}

\section*{Acknowledgement}
Z. Wang is partially supported by the National Key R\&D Program of China, Project Number 2021YFA1002800, NSFC grant No.12171009, Young
Elite Scientist Sponsorship Program by China Association for Science and Technology
(CAST) No. YESS20200028. 
The work of V. Ilin and J. Hu was partially supported by AFOSR grant FA9550-21-1-0358 and DOE grant DE-SC0023164.

\printbibliography

\end{document}